\documentclass[12pt]{article}

\usepackage{amsmath,amssymb,amsthm,mathtools}
\usepackage{enumitem}
\usepackage{geometry}
\usepackage{microtype}
\usepackage{mathrsfs}

\newtheorem{theorem}{Theorem}[section]
\newtheorem{proposition}[theorem]{Proposition}
\newtheorem{lemma}[theorem]{Lemma}
\newtheorem{corollary}[theorem]{Corollary}

\theoremstyle{definition}
\newtheorem{definition}[theorem]{Definition}

\theoremstyle{remark}
\newtheorem*{remark}{Remark}

\newcommand{\F}{\mathbb F}

\newcommand{\gr}{\operatorname{gr}}
\newcommand{\Span}{\operatorname{Span}}
\newcommand{\wt}{\operatorname{wt}}
\newcommand{\Lie}{\mathrm{Lie}}

\newcommand{\Omegatopp}{\Omega_{\mathrm{top},p}}
\newcommand{\GL}{GL}

\title{Lower and Upper Lie Nilpotency Indices of Modular Group Algebras}

\author{
Joe Gildea\\
Department of Computing Science and Mathematics,\\
School of Informatics and Creative Arts,\\
Dundalk Institute of Technology, Ireland\\
\texttt{gildeajoe@gmail.com}
}

\date{}

\begin{document}

\maketitle

\begin{abstract}
We resolve the remaining cases of a longstanding problem on the lower
and upper Lie nilpotency indices of modular group algebras. Equality
of the two indices has been known in characteristics greater than three
since 1992; we establish it in the remaining characteristics two and
three, thereby completing the problem for finite groups of prime-power
order over finite fields. The proof gives a uniform argument for every
prime and introduces a terminal restricted Poincaré–Birkhoff–Witt construction that
converts the lower-Lie problem into a chronological reachability
argument, with no structural hypothesis on the derived subgroup.
\end{abstract}



\section{Introduction}
\label{sec:introduction}

Let \(K\) be a field of characteristic \(p>0\) and let \(G\) be a group
such that the group algebra \(KG\) is Lie nilpotent. Regarding \(KG\) as
a Lie algebra under the commutator \( [x,y]=xy-yx \), one obtains two
natural Lie filtrations, the lower and upper Lie powers, and hence the
corresponding lower and upper Lie nilpotency indices \(t_L(KG)\) and
\(t^L(KG)\). The study of these indices is rooted in the classical
theory of dimension subgroups and modular group algebras. Jennings'
description of the dimension filtration of a modular group algebra
\cite{Jennings1941} provided the basic connection between the
group theoretic filtration and the structure of the augmentation ideal.
The theory of Lie dimension subgroups was subsequently developed further
by Passi and Sehgal \cite{PassiSehgal1975} and by Shalev
\cite{Shalev1991}, among others, and provides the natural framework for
the upper Lie nilpotency index.

A central problem in this theory is to determine when the lower and
upper Lie nilpotency indices coincide. One always has $t_L(KG)\leq t^L(KG)$,
but the reverse inequality is substantially more delicate. Bhandari and
Passi \cite{BhandariPassi1992} proved equality in characteristic
\(p>3\). Subsequent work has determined the indices, or established
their equality, under a variety of additional hypotheses and for
important extremal or prescribed values. This includes the study of
maximal and almost maximal Lie nilpotency indices by Bovdi and Spinelli
\cite{BovdiSpinelli2005} and Bovdi, Juhász and Spinelli
\cite{BovdiJuhaszSpinelli2006}, further investigations of Lie
nilpotency indices and Lie dimension subgroups by Spinelli
\cite{Spinelli2009}, Bovdi and Srivastava \cite{BovdiSrivastava2010},
and a series of later results concerning particular upper and lower Lie
nilpotency indices and related classification problems
\cite{SahaiSharmaSiwach2017,SahaiSharan2018,SahaiSiwachSharma2018,
SahaiSharan2019,SahaiSharan2021,BhattChandra2022}. These results resolve
many significant special cases, while leaving the equality problem in
the remaining characteristics without a general structural hypothesis
on \(G'\).

The purpose of the present paper is to establish the equality in full
generality for finite \(p\) groups over finite fields of characteristic
\(p\). More precisely, we prove $t_L(KG)=t^L(KG)$
for every finite field \(K\) of characteristic \(p\) and every finite
\(p\) group \(G\). No additional structural assumption is imposed on
the derived subgroup \(G'\). In particular, this settles the remaining
characteristics \(p=2\) and \(p=3\).

The proof differs from the predominantly structural approach of the
existing literature. Rather than deriving the lower index from a
classification or from additional structural restrictions on \(G'\),
we retain the classical Lie dimension filtration as the description of
the upper index and construct directly a lower Lie commutator of the
required length. The bridge between the two is a terminal restricted
PBW construction. The dimension data determine a terminal PBW target
of Jennings weight, and the main argument proves that this target is
reachable by one genuine chronological lower Lie history. The
construction is uniform in \(p\) and requires no additional structural
hypothesis on \(G'\). We first prove the result over the prime field
\(\F_p\) and then pass to every finite extension of \(\F_p\) by scalar
extension.

The starting point is the classical upper Lie dimension filtration. If
\[
    D_{(m),K}(G):=G\cap(1+KG^{(m)}),
\]
then the standard containments
\[
    [D_{(r),K}(G),D_{(s),K}(G)]
        \subseteq D_{(r+s),K}(G),
    \qquad
    D_{(r),K}(G)^p\subseteq D_{(pr),K}(G)
\]
hold \cite{PassiSehgal1975,Shalev1991}. If
\[
    p^{d_{(m)}}=
    |D_{(m),K}(G):D_{(m+1),K}(G)|,
\]
then Jennings' formula gives
\[
    t^L(KG)
    =
    2+(p-1)\sum_{m\geq1}m\,d_{(m+1)}.
\]
Thus the upper index is determined by the successive homogeneous layers
of the Lie dimension filtration. By contrast, the lower index is
governed by actual iterated Lie commutators. The central problem is
therefore to pass from the static graded information determining
\(t^L(KG)\) to one chronological lower Lie process of the full required
length.

We carry out this passage first for \(R=\F_pG\). Put
\[
    V_m:=D_{(m+1)},\qquad
    E_m:=V_m/V_{m+1},\qquad
    E:=\bigoplus_{m\geq1}E_m,
\]
write \(d_m:=\dim_{\F_p}E_m\), and set
\[
    N_p:=(p-1)\sum_{m\geq1}m\,d_m.
\]
The shifted filtration satisfies
\[
    [V_a,V_b]\subseteq V_{a+b+1},
    \qquad
    V_a^p\subseteq V_{pa},
\]
and \(E\) is an abelian graded restricted Lie algebra. Jennings'
formula becomes
\[
    t^L(R)=N_p+2.
\]
Choose an ordered homogeneous basis
\(\mathcal B=(e_1,\ldots,e_d)\) of \(E\). By the restricted
Poincaré Birkhoff Witt theorem, the monomials
\[
    e_1^{\alpha_1}\cdots e_d^{\alpha_d},
    \qquad
    0\leq\alpha_i\leq p-1,
\]
form a basis of \(\mathcal A:=u(E)\). Define the terminal monomial
\[
    \Omegatopp:=e_1^{p-1}\cdots e_d^{p-1}.
\]
Then \(\Omegatopp\neq0\) and
\[
    \wt(\Omegatopp)
    =
    (p-1)\sum_{m\geq1}m\,d_m
    =
    N_p.
\]
Thus the upper Lie dimension data determine a single terminal
restricted PBW monomial of precisely the weight required by Jennings'
formula. For background on restricted Lie algebras and restricted
enveloping algebras, see
\cite{Jacobson1962,StradeFarnsteiner1988}.

The central reduction is to show that this terminal PBW state is the
graded shadow of one genuine chronological lower Lie history. We
formulate chronological routing in \(u(E)\) so that a routing word
records the graded output of successive operations, while genuineness
requires the entire history to arise from a single actual left normed
commutator. If \(W\) is such a genuine chronological routing word and
$[\Omegatopp]W(1)\neq0$, then homogeneity forces \(W\) to have routing weight \(N_p\), while
genuineness produces a nonzero \((N_p+1)\) fold lower Lie commutator.
Consequently, $R^{[N_p+1]}\neq0$, $t_L(R)\geq N_p+2=t^L(R)$. The general containment \(A^{[n]}\subseteq A^{(n)}\) gives the reverse
inequality. The equality problem is therefore reduced to terminal PBW
reachability.

To prove this reachability, restricted PBW monomials are encoded by
occupancy vectors, with a nonnegative occupancy deficit measuring the
distance from the terminal vector \((p-1,\ldots,p-1)\). We construct a
fixed homogeneous routing basis whose nonprimitive directions carry
stored provenance through two operations, PUSH and restricted
\(p\) POWER. Tracing this provenance backwards strictly decreases
homogeneous weight. Since PUSH acts as a derivation and restricted
normalization can produce competing PBW terms, a protection mechanism
isolates the distinguished routing contribution and prevents
cancellation.

The local constructions of reconstruction, FILL, and POWER depend on
one another and are therefore proved simultaneously by strong induction
on homogeneous weight. They are then assembled into the
BUILD and REBUILD mechanism. BUILD creates new certified occupancy,
while REBUILD restores any previously certified occupancy spent in
producing it. Restoration follows stored provenance backwards, and a
well founded multiset of debt weights guarantees termination. The
resulting PBW reachability theorem states that every nonterminal
completed certified state admits a finite genuine chronological
continuation which preserves all previously certified occupancy and
strictly decreases the occupancy deficit.

Since the deficit is a nonnegative integer, repeated reachability
terminates at a completed certified terminal state. The successive
continuations concatenate to one genuine chronological routing history.
After rebasing the resulting terminal routing basis monomial to
\(\Omegatopp\), its coefficient remains nonzero. The terminal PBW
reduction therefore gives $t_L(\F_pG)\geq t^L(\F_pG)$. Together with the reverse inequality,
$t_L(\F_pG)=t^L(\F_pG)=N_p+2$. Finally, scalar extension commutes with both Lie power constructions
and preserves their first vanishing terms, giving $t_L(KG)=t^L(KG)=N_p+2$ for every finite field \(K\) of characteristic \(p\).

The organisation of the paper follows this construction.
Section~\ref{sec:pbw-chronological-framework} develops the restricted
PBW and chronological framework and reduces the difficult inequality
to terminal PBW reachability.
Section~\ref{sec:routing-protection} constructs the fixed routing basis
and the protection and noncancellation machinery.
Section~\ref{sec:simultaneous-reconstruction-fill-power} establishes
reconstruction, Certified FILL, and Certified POWER simultaneously by
strong induction on homogeneous weight.
Section~\ref{sec:build-rebuild} develops the terminating BUILD and
REBUILD mechanism, and Section~\ref{sec:pbw-reachability} proves PBW
reachability.
Finally, Section~\ref{sec:terminal-completion-equality} iterates
reachability to the terminal state, proves the required lower bound,
and completes the proof over finite fields by scalar extension.

\section{PBW and Chronological Framework}
\label{sec:pbw-chronological-framework}

The purpose of this section is to formulate the terminal PBW target and
the chronological framework in which it will be reached.  We first
record the coefficient and homogeneous-weight language for the
restricted PBW basis of \(\mathcal A=u(E)\), and then define
chronological routing and the genuineness condition which connects a
graded routing history to an actual lower-Lie commutator in \(R\).
This reduces the difficult inequality \(t_L(R)\geq t^L(R)\) to the
construction of one genuine chronological history whose complete
graded output has nonzero coefficient at the terminal PBW monomial
\(\Omegatopp\).

We then encode restricted PBW monomials by occupancy vectors.  The
resulting occupancy deficit measures the distance from a certified PBW
state to the terminal state and will provide the finite progress
measure used in the subsequent routing and BUILD REBUILD
construction. We first identify the restricted PBW monomial whose weight realizes the
quantity appearing in Jennings' formula.  This will be the terminal
state of the subsequent construction.

Choose an ordered homogeneous basis \(\mathcal B=(e_1,\ldots,e_d)\) of \(E\), and let \(\mathcal A:=u(E)\). For a homogeneous element \(x\in E_m\), write \(\wt(x):=m\) for its upper-Lie weight, and extend this additively to restricted PBW monomials:
\[
    \wt\!\left(e_1^{\varepsilon_1}\cdots e_d^{\varepsilon_d}\right)
    :=
    \sum_{i=1}^{d}\varepsilon_i\,\wt(e_i).
\]

By the restricted PBW theorem, the monomials \(e_1^{\varepsilon_1}\cdots e_d^{\varepsilon_d}\), \(0\leq\varepsilon_i\leq p-1\), form an \(\F_p\)-basis of \(\mathcal A\). Define \(\Omegatopp:=e_1^{p-1}\cdots e_d^{p-1}\). By restricted PBW, \(\Omegatopp\neq0\). Moreover,
\[
    \wt(\Omegatopp)
    =
    \sum_{i=1}^{d}(p-1)\wt(e_i)
    =
    (p-1)\sum_{m\geq1}m\,d_m
    =
    N_p
    =
    t^L(R)-2.
\]

Thus \(\Omegatopp\) is a nonzero restricted PBW monomial of precisely the weight required to detect the last possible nonzero stage relevant to the lower Lie nilpotency index. Our task will therefore be to produce \(\Omegatopp\) with nonzero coefficient as the graded shadow of a single actual commutator history. To track this distinguished monomial through the construction, we use the following coefficient notation.

\begin{definition}
\label{def:terminal-pbw-coefficient}
Let $M  =e_1^{\varepsilon_1} \cdots e_d^{\varepsilon_d}$, $0\leq\varepsilon_i\leq p-1$, be an ordered restricted PBW monomial.
For $X\in\mathcal A$, write its unique restricted PBW expansion as
\[
        X
        =
        \sum_{\epsilon\in\{0,\ldots,p-1\}^{d}}
        c_\epsilon\,
        e_1^{\epsilon_1}\cdots e_d^{\epsilon_d},
        \qquad
        c_\epsilon\in\F_p.
\]

The coefficient of \(M\) in \(X\) is denoted by $[M]X$.
\end{definition}

\begin{lemma}
\label{lem:terminal-pbw-coefficient-detection}

Let $X\in\mathcal A$. If $[M]X\neq0$ for some restricted PBW monomial \(M\), then $X\neq0$. In particular,
$X\neq0$ if $[\Omegatopp]X\neq0$.
\end{lemma}

\begin{proof}
The restricted PBW monomials form an \(\F_p\)-basis of $\mathcal A$. Hence their coefficients in the PBW expansion of any element are uniquely determined.
If $X=0$, then every restricted PBW coefficient of \(X\) is zero. Therefore $[M]X\neq0$ for even one restricted PBW monomial \(M\) forces
$ X\neq0$. Taking $M=\Omegatopp$ gives the final assertion.
\end{proof}

Nonvanishing of a restricted PBW coefficient therefore detects
nonvanishing of the complete graded output.  To relate such an output
to the length of a chronological commutator history, we next record
the homogeneous upper-Lie grading on \(\mathcal A\).

\begin{definition}
\label{def:homogeneous-upper-lie-components}

For each \(n\geq 0\), define
\[
        \mathcal A_n
        :=
        \operatorname{span}_{\F_p}
        \left\{
        e_1^{\varepsilon_1}\cdots e_d^{\varepsilon_d}
        :
        \sum_{i=1}^{d}
        \varepsilon_i\,\wt(e_i)=n
        \right\}.
\]

Thus \(\mathcal A_n\) is the homogeneous component of upper-Lie
weight \(n\) in $ \mathcal A=u(E)$.

\end{definition}

We now pass from the static PBW grading to the chronological
construction that will be used throughout the proof.  A routing word
records the successive graded operations, while its routing weight
records their total homogeneous weight increment.

\begin{definition}
\label{def:chronological-routing-word}
A \emph{chronological routing word} is a finite composition $W = T_rT_{r-1}\cdots T_1$, where every \(T_j\) is an elementary chronological routing step. If the routing weight of \(T_j\) is \(q_j\), define the total routing weight of \(W\) by $\wt_{\mathrm{route}}(W) := q_1+\cdots+q_r$. Its complete graded output from the unit state is denoted by $W(1)$.
\end{definition}

\begin{lemma}
\label{lem:chronological-routing-homogeneous}
For every chronological routing word \(W\), $W(1)\in \mathcal A_{\wt_{\mathrm{route}}(W)}$.
\end{lemma}

\begin{proof}
Since \(1\in\mathcal A_0\), the assertion follows by adding the
homogeneous weight increments of the successive routing operations.
\end{proof}

A chronological routing word by itself records only a graded
construction.  For the lower Lie series we require more: the entire
graded history must arise from one actual chronological sequence of
Lie brackets in \(R\).  The following definition records precisely
this additional requirement.

\begin{definition}
\label{def:genuine-chronological}
A chronological routing word \(W\), of routing weight \(n\), is
\emph{genuine chronological} if its complete graded history is the
leading graded history of a single actual left-normed commutator
$C_{n+1}= [x_1,x_2,\ldots,x_{n+1}]$, $x_i\in R$. Thus \(W(1)\) is the image of \(C_{n+1}\) in the corresponding
associated-graded component.
\end{definition}

\begin{proposition}
\label{prop:terminal-pbw-reduction}
Suppose that there exists a genuine chronological routing word \(W\) such that $[\Omegatopp]W(1)\neq0$. Then $R^{[N_p+1]}\neq0$ and consequently $t_L(R)\geq t^L(R)$.
\end{proposition}

\begin{proof}
Let $n:=\wt_{\mathrm{route}}(W)$. By Lemma~\ref{lem:chronological-routing-homogeneous}, $W(1)\in\mathcal A_n$. Since \([\Omegatopp]W(1)\neq0\), the terminal PBW monomial occurs
in this homogeneous component. Hence $n=\wt(\Omegatopp)=N_p$. Since \(W\) is genuine chronological, its complete graded history is the leading graded history of
an actual left-normed commutator $C_{N_p+1} = [x_1,\ldots,x_{N_p+1}]$, and \(W(1)\) is the image of \(C_{N_p+1}\) in the corresponding associated-graded component. The nonzero coefficient $[\Omegatopp]W(1)\neq0$ implies \(W(1)\neq0\), and therefore \(C_{N_p+1}\neq0\).
Thus $R^{[N_p+1]}\neq0$. It follows that $t_L(R)\geq N_p+2=t^L(R)$.
\end{proof}

Proposition~\ref{prop:terminal-pbw-reduction} reduces the difficult
inequality to a concrete reachability problem: it is enough to produce
one genuine chronological routing word whose complete graded output
has nonzero coefficient at the terminal restricted PBW monomial.

We now formulate this reachability problem as a finite occupancy
problem in \(u(E)\).  A restricted PBW monomial is encoded by its
exponent vector, and the occupancy deficit measures the distance of
this vector from the terminal vector
\((p-1,\ldots,p-1)\).  Terminal reachability will therefore amount to
reducing a nonnegative integer-valued deficit to zero while preserving
the occupancies already certified.

\begin{definition}
\label{def:general-p-pbw-state}
A \emph{restricted PBW state} is an occupancy vector
$\nu=(\nu_1,\ldots,\nu_d)\in\{0,\ldots,p-1\}^d$
relative to a fixed ordered homogeneous basis
\(\mathcal C=(c_1,\ldots,c_d)\) of \(E\).
Its distinguished restricted PBW monomial is
$M(\nu)=c_1^{\nu_1}\cdots c_d^{\nu_d}$.
A direction \(c_i\) is occupied precisely when \(\nu_i>0\).
\end{definition}

\begin{definition}
\label{def:general-p-completed-certified-state}
Let \(X\in u(E)\) be the complete graded output of one genuine chronological routing history. A restricted PBW state \(\nu\) relative to the fixed ordered homogeneous basis \(\mathcal C=(c_1,\ldots,c_d)\) is \emph{completed certified in \(X\)} if $[M(\nu)]_{\mathcal C}X\in\F_p^\times$.
\end{definition}

\begin{definition}
\label{def:general-p-local-deficit}
Let $\nu=(\nu_1,\ldots,\nu_d)\in\{0,\ldots,p-1\}^d$ be a completed certified restricted PBW state relative to the fixed ordered homogeneous basis \(\mathcal C=(c_1,\ldots,c_d)\) of \(E\). Define its \emph{total completed occupancy} by
\[
    |\nu|_{\mathrm{occ}}
    :=
    \sum_{i=1}^{d}\nu_i.
\]

Define its \emph{local PBW occupancy deficit} by
\[
    \Delta_p(\nu)
    :=
    \sum_{i=1}^{d}\bigl((p-1)-\nu_i\bigr)
    =
    (p-1)d-|\nu|_{\mathrm{occ}}.
\]

Since \(0\leq\nu_i\leq p-1\), one has $0\leq\Delta_p(\nu)\leq(p-1)d$.

\end{definition}

\begin{definition}
\label{def:general-p-terminal-state}
A completed certified restricted PBW state \(\nu=(\nu_1,\ldots,\nu_d)\) relative to \(\mathcal C\) is called \emph{terminal} if $\nu=(p-1,\ldots,p-1)$. Equivalently,
$M(\nu) = c_1^{p-1}\cdots c_d^{p-1}$.
\end{definition}

\begin{proposition}
\label{prop:general-p-terminal-characterization}
Let $\nu=(\nu_1,\ldots,\nu_d)$ be a completed certified restricted PBW state relative to
\(\mathcal C\). Then the following are equivalent:
\begin{itemize}
\item[(1)] \(\nu\) is terminal;
\item[(2)] \(\nu=(p-1,\ldots,p-1)\);
\item[(3)] \( |\nu|_{\mathrm{occ}}=(p-1)d\);
\item[(4)] \(\Delta_p(\nu)=0\).
\end{itemize}
\end{proposition}

\begin{proof}
By Definition~\ref{def:general-p-terminal-state}, conditions \textup{(1)} and \textup{(2)} are equivalent. If \(\nu=(p-1,\ldots,p-1)\), then $|\nu|_{\mathrm{occ}}=(p-1)d$ and hence
$\Delta_p(\nu)= (p-1)d-|\nu|_{\mathrm{occ}} =0$. Conversely, if \(\Delta_p(\nu)=0\), then $\sum_{i=1}^{d}\bigl((p-1)-\nu_i\bigr)=0$. Every summand is a nonnegative integer, so every summand is zero. Thus \(\nu_i=p-1\) for every \(i\), and therefore \(\nu=(p-1,\ldots,p-1)\). The remaining equivalence follows from $\Delta_p(\nu)=(p-1)d-|\nu|_{\mathrm{occ}}$.
\end{proof}

The preceding reduction now gives a precise finite target for the
remainder of the argument.  Starting from a completed certified PBW
state, it is enough to construct genuine chronological continuations
which preserve the occupancies already certified while increasing the
total completed occupancy.  Equivalently, each successful continuation
must strictly decrease the nonnegative integer
\(\Delta_p(\nu)\).  Repetition of such progress must eventually reach
\(\Delta_p(\nu)=0\), which by
Proposition~\ref{prop:general-p-terminal-characterization} is exactly
the terminal state.  At that point the distinguished monomial is the
terminal restricted PBW monomial, and
Proposition~\ref{prop:terminal-pbw-reduction} converts its nonzero
coefficient along a genuine chronological history into the required
lower-Lie nonvanishing.

The remaining task is therefore constructive: we must develop routing
operations which create and propagate the required PBW directions
without destroying coefficients that have already been certified.
The next section develops the routing, ancestry, protection, and
noncancellation machinery needed for this purpose.


\section{Routing and Protection Machinery}
\label{sec:routing-protection}

This section develops the routing and protection machinery needed for the
chronological construction. We first construct a fixed homogeneous routing
basis of \(E\), with each basis direction assigned a stored ancestry from
\(E_1\) through PUSH and restricted POWER. We then develop the
noncancellation mechanism that protects distinguished routing contributions
from competing descendant terms and restricted PBW normalization. Finally,
we introduce the CREATE, FILL, and POWER primitives that will be used in the
simultaneous reconstruction argument of the next section.

\subsection{Routing, ancestry, and provenance}
\label{subsec:routing-ancestry-provenance}

We begin by constructing a single homogeneous routing basis of \(E\), with
each basis direction assigned a fixed ancestry from \(E_1\) through transport
(PUSH) and restricted \(p\)-POWER. This stored ancestry will determine the
routes along which later chronological constructions produce new basis
directions. We first establish the algebraic properties of transport needed
for this construction.

\begin{lemma}
\label{lem:general-p-commutator-power-congruence}

Let \(v\in V_m\) and \(g\in G\). Then $[v^p,g] \equiv [v,g]^p \pmod{V_{pm+2}}$.
\end{lemma}

\begin{proof}
Put $c=[v,g]\in V_{m+1}$. Repeated use of the commutator identity $[ab,g]=[a,g]^b[b,g]$ gives $[v^p,g] = c^{v^{p-1}}c^{v^{p-2}}\cdots c^vc$. Applying the Hall Petrescu collection formula
\cite{Hall1934,Dixonetall1999} to this product, the principal factor is \(c^p\). Every remaining collected factor is a basic commutator in \(c\) and \(v\). If such a commutator has weight \(k\), with \(2\leq k\leq p-1\), then its exponent in the collection is divisible by \(p\); this is the standard binomial-divisibility consequence of Hall Petrescu collection, since
\[
    p\mid \binom{p}{j},
    \qquad 1\leq j\leq k<p.
\]
The remaining collected commutators have weight at least \(p\). Let \(h\) be a nonprincipal basic commutator of weight \(k\geq2\). Since every such \(h\) contains at least one occurrence of \(c\in V_{m+1}\), while every remaining entry lies in \(V_m\), repeated use of $[V_a,V_b]\subseteq V_{a+b+1}$ gives $h\in V_{k(m+1)}$. Suppose first that \(2\leq k\leq p-1\). If \(e_h\) is the exponent of \(h\) in the Hall Petrescu collection, then \(p\mid e_h\), say \(e_h=pq\). Hence $h^{e_h}=(h^q)^p$. Since \(h^q\in V_{k(m+1)}\), the \(p\)-power containment gives $h^{e_h} \in V_{pk(m+1)} \subseteq V_{pm+2}$. If \(k\geq p\), then already
\[
        h
        \in
        V_{k(m+1)}
        \subseteq
        V_{p(m+1)}
        =
        V_{pm+p}
        \subseteq
        V_{pm+2},
\]
and hence \(h^{e_h}\in V_{pm+2}\). Thus every nonprincipal factor in the Hall Petrescu collection lies in \(V_{pm+2}\). Therefore $[v^p,g] \equiv c^p \pmod{V_{pm+2}}$. Since \(c=[v,g]\), we obtain $[v^p,g] \equiv [v,g]^p \pmod{V_{pm+2}}$.
\end{proof}

\begin{definition}
\label{def:general-p-transport}
For \(m\geq1\) and \(g\in G\), the \emph{transport map} $D_g:E_m\longrightarrow E_{m+1}$ is defined by
$D_g(vV_{m+1})=[v,g]V_{m+2}$, $v\in V_m$. We refer to the operation induced by \(D_g\) as a \emph{PUSH}.
\end{definition}

\begin{lemma}
\label{lem:general-p-transport-well-defined}
For every \(m\geq1\) and \(g\in G\), the transport map $D_g:E_m\longrightarrow E_{m+1}$, $D_g(vV_{m+1})=[v,g]V_{m+2}$, is a well-defined \(\F_p\)-linear map. Moreover,
$D_g(E_m)\subseteq E_{m+1}$, so \(D_g\) has homogeneous weight degree \(+1\).

\end{lemma}

\begin{proof}
Recall that $E_m=V_m/V_{m+1}$. Let \(v,w\in V_m\) represent the same element of \(E_m\). Then $vV_{m+1}=wV_{m+1}$, so $vw^{-1}\in V_{m+1}$. Write $v=uw$, $u\in V_{m+1}$. Using the commutator identity $[ab,g]=[a,g]^b[b,g]$, we obtain $[v,g]=[uw,g]=[u,g]^w[w,g]$. Since \(u\in V_{m+1}\), the standard filtration containment $[V_{m+1},G]\subseteq V_{m+2}$ gives
$[u,g]\in V_{m+2}$. Because \(V_{m+2}\) is normal in \(G\), also $[u,g]^w\in V_{m+2}$. Hence $[v,g][w,g]^{-1}\in V_{m+2}$, and therefore $[v,g]V_{m+2} = [w,g]V_{m+2}$.
Thus \(D_g\) is independent of the chosen representative of \(vV_{m+1}\), and so is well defined. Let \(v,w\in V_m\). Using $[ab,g]=[a,g]^b[b,g]$, we have
$[vw,g]=[v,g]^w[w,g]$. Since $[v,g]\in V_{m+1}$ and
$[[v,g],w]\in [V_{m+1},V_m]\subseteq V_{m+2}$, we have
$[v,g]^w\equiv[v,g]\pmod{V_{m+2}}$. Hence
\[
[vw,g]V_{m+2}
=
[v,g][w,g]V_{m+2}.
\]
Thus \(D_g\) is a homomorphism between the elementary abelian \(p\)-groups
\(E_m\) and \(E_{m+1}\), and therefore is \(\F_p\)-linear. Finally, if \(v\in V_m\), then $[v,g]\in [V_m,G]\subseteq V_{m+1}$.
Therefore
\[
D_g(vV_{m+1})
=
[v,g]V_{m+2}
\in
V_{m+1}/V_{m+2}
=
E_{m+1}.
\]
Thus $D_g(E_m)\subseteq E_{m+1}$, so \(D_g\) has homogeneous weight degree \(+1\).
\end{proof}

\begin{lemma}
\label{lem:general-p-transport-power-zero}
Let $x\in E_m$ be homogeneous and let \(g\in G\). Then $D_g(x^{[p]})=0$ in $E_{pm+1}$.
\end{lemma}

\begin{proof}
Let \(x\in E_m\), and choose \(v\in V_m\) representing \(x\). Then \(x^{[p]}\in E_{pm}\) is represented by \(v^p\in V_{pm}\). Hence $D_g(x^{[p]}) = [v^p,g]V_{pm+2} \in E_{pm+1}$. By Lemma~\ref{lem:general-p-commutator-power-congruence}, $[v^p,g] \equiv [v,g]^p \pmod{V_{pm+2}}$. Put \(c=[v,g]\in V_{m+1}\). Then $c^p\in V_{p(m+1)}$. Since \(p\geq2\),
$p(m+1)=pm+p\geq pm+2$, and hence $V_{p(m+1)} \subseteq V_{pm+2}$, since \(V_{i+1}\subseteq V_i\) for all \(i\). Therefore $[v,g]^p=c^p\in V_{pm+2}$. It follows that
$[v^p,g] \equiv 1 \pmod{V_{pm+2}}$, so $[v^p,g]\in V_{pm+2}$. Therefore its image in $E_{pm+1}=V_{pm+1}/V_{pm+2}$ is zero. Thus $D_g(x^{[p]})=0$.
\end{proof}

\begin{proposition}
\label{prop:general-p-push-derivation}
For every \(g\in G\), the degree-one map $D_g:E\longrightarrow E$ extends uniquely to a degree-one derivation $\widetilde D_g:T(E)\longrightarrow T(E)$. Moreover, \(\widetilde D_g\) induces a unique degree-one derivation $\overline D_g:u(E)\longrightarrow u(E)$. Thus $\overline D_g(ab)=\overline D_g(a)b+a\overline D_g(b)$
for all $a,b\in u(E)$.
\end{proposition}

\begin{proof}
Let $T(E)$ be the tensor algebra on \(E\). The linear map $D_g:E\longrightarrow E$ extends uniquely to a derivation $\widetilde D_g:T(E)\longrightarrow T(E)$. We verify that
\(\widetilde D_g\) preserves the defining ideal of the restricted enveloping algebra. Since \(E\) is abelian, the ordinary Lie relations are $xy-yx=0$. Applying \(\widetilde D_g\) gives
\[
\widetilde D_g(xy-yx)
=
D_g(x)y+xD_g(y)-D_g(y)x-yD_g(x),
\]
which vanishes modulo the commutativity relations. Now consider the restricted relation $x^p-x^{[p]}$. Modulo
commutativity,
\[
\widetilde D_g(x^p)
=
\sum_{j=0}^{p-1}x^jD_g(x)x^{p-1-j}
=
p\,x^{p-1}D_g(x)
=
0
\]
in characteristic \(p\). Write $x=\sum_m x_m$, where each
$x_m\in E_m$ is homogeneous. Since \(E\) is abelian as a restricted
Lie algebra, the restricted \(p\)-map is additive, and hence
\[
x^{[p]}=\sum_m x_m^{[p]}.
\]
Therefore, by Lemma~\ref{lem:general-p-transport-power-zero},
\[
D_g(x^{[p]})
=
\sum_m D_g(x_m^{[p]})
=
0.
\]
Hence $\widetilde D_g\left(x^p-x^{[p]}\right)=0$ modulo the defining ideal. Thus \(\widetilde D_g\) descends to a
derivation $\overline D_g:u(E)\longrightarrow u(E)$. Since \(u(E)\) is generated as an algebra by \(E\), this extension is
unique. Finally, $D_g(E_m)\subseteq E_{m+1}$, so the resulting
derivation \(\overline D_g\) has homogeneous weight degree \(+1\).
\end{proof}

\begin{remark}
The derivation $\overline D_g:u(E)\longrightarrow u(E)$ restricts to the
original transport map $D_g:E\longrightarrow E$. Henceforth we write \(D_g\)
for both maps; the intended meaning is determined by whether the argument
lies in \(E\) or in \(u(E)\).
\end{remark}

We now organize the homogeneous directions of \(E\) according to the two
operations available in the graded restricted Lie algebra: transport \(D_g\)
and restricted \(p\)-POWER.

\begin{definition}
\label{def:general-p-routing-compatible-ancestry}

Let $0\neq u\in E_m$ be homogeneous. A \emph{routing-compatible ancestry} for \(u\) is a finite chain
\[
y_0 \longmapsto y_1 \longmapsto \cdots \longmapsto y_s=u
\]
such that $y_0\in E_1$ and every edge is of one of the two forms $y_{j+1}=D_{g_j}(y_j)$
for some \(g_j\in G\), or $y_{j+1}=y_j^{[p]}$. The first type is called a \emph{PUSH edge}; the second is called a
\emph{POWER edge}.
\end{definition}

\begin{lemma}
\label{lem:general-p-routing-generation}
Let
\[
    U
    \subseteq
    E=\bigoplus_{m\geq1}E_m
\]
be the smallest graded \(\F_p\)-subspace containing \(E_1\) and closed under every transport map $D_g:E_m\longrightarrow E_{m+1}$, $g\in G$, and under restricted POWER
$P_m:E_m\longrightarrow E_{pm}$, $P_m(x)=x^{[p]}$. Then $U=E$. More precisely, for every \(m\geq2\),
\[
    E_m
    =
    \sum_{g\in G}D_g(E_{m-1})
    +
    \begin{cases}
    P_{m/p}(E_{m/p}), & p\mid m,\\[2mm]
    0, & p\nmid m.
    \end{cases}
\]
Consequently every homogeneous component of \(E\) is spanned by
directions obtained from \(E_1\) by finitely many PUSH and restricted
POWER operations.
\end{lemma}

\begin{proof}
Recall that $V_m=D_{(m+1)}(G)$, $E_m=V_m/V_{m+1}$. The standard recursion for the upper Lie dimension subgroups (\cite{Passi1979}) is $D_{(m+1)}(G) = [D_{(m)}(G),G]\, D_{(\lceil m/p\rceil+1)}(G)^p$, $m\geq2$. In the shifted notation this becomes $V_m = [V_{m-1},G]\, V_{\lceil m/p\rceil}^{\,p}$. Put $r=\left\lceil\frac{m}{p}\right\rceil$. Pass to the quotient
$E_m=V_m/V_{m+1}$. Since \(E_m\) is elementary abelian, the image of the product on the right is the sum of the images of its two factors. The image of \([V_{m-1},G]\) is
$\sum_{g\in G}D_g(E_{m-1})$, because, for \(v\in V_{m-1}\), $D_g(vV_m) = [v,g]V_{m+1}$. It remains to determine the image of \(V_r^p\). By the standard filtration containment
$V_r^p\subseteq V_{pr}$. Suppose first that \(p\nmid m\). Then $pr = p\left\lceil\frac{m}{p}\right\rceil >m$, and hence, since the indices are integral,
$pr\geq m+1$. Therefore $V_r^p \subseteq V_{pr} \subseteq V_{m+1}$, so the \(p\)-power factor has zero image in \(E_m\). Now suppose that \(p\mid m\), say \(m=pr\). Then
$V_r^p\subseteq V_{pr}=V_m$, and its image in \(V_m/V_{m+1}\) is precisely the restricted
POWER image $P_r(E_r)=P_{m/p}(E_{m/p})$.
Thus
\[
    E_m
    =
    \sum_{g\in G}D_g(E_{m-1})
    +
    \begin{cases}
    P_{m/p}(E_{m/p}), & p\mid m,\\[2mm]
    0, & p\nmid m.
    \end{cases}
\]

We now prove \(U=E\) by induction on homogeneous weight. By definition, $E_1\subseteq U$. Let \(m\geq2\), and suppose inductively that $E_j\subseteq U$ $(1\leq j<m)$. Since \(m-1<m\), we have \(E_{m-1}\subseteq U\). If \(p\mid m\), then \(m/p<m\), so also \(E_{m/p}\subseteq U\). Because \(U\) is closed under every \(D_g\) and under restricted POWER, the displayed decomposition of \(E_m\) gives $E_m\subseteq U$. Hence \(E_m\subseteq U\) for every \(m\geq1\), and therefore $U=E$.
\end{proof}

\begin{lemma}
\label{lem:general-p-fixed-routing-basis}
There exists an ordered homogeneous basis
$\mathcal B_{\mathrm{rt}}=(b_1,\ldots,b_d)$ of \(E\) with the following properties.

\begin{enumerate}
\item The basis elements lying in \(E_1\) may be chosen to be images of genuine group commutators.
\item For every basis element \(b_j\notin E_1\), there exists an earlier basis element \(b_i\), \(i<j\), such that either $b_j=D_g(b_i)$ for some \(g\in G\), or $b_j=b_i^{[p]}$.
\item Consequently every \(b_j\) has a stored routing-compatible ancestry all of whose vertices belong to the single fixed basis \(\mathcal B_{\mathrm{rt}}\).
\end{enumerate}
\end{lemma}

\begin{proof}
Choose a homogeneous basis of \(E_1\), consisting of images of genuine group commutators, and suppose that homogeneous vectors \(b_1,\ldots,b_s\) have been chosen linearly independently. Put
$U_s:=\Span_{\F_p}\{b_1,\ldots,b_s\}$. If \(U_s\neq E\), then Lemma~\ref{lem:general-p-routing-generation} implies that \(U_s\) is not closed under all \(D_g\) and restricted
POWER. Hence, for some homogeneous \(x\in U_s\), $D_g(x)\notin U_s$ or $x^{[p]}\notin U_s$. Writing \(x=\sum_i\lambda_i b_i\) in its homogeneous component, and using the \(\F_p\)-linearity of \(D_g\) and of \(x\mapsto x^{[p]}\), there is some \(i\leq s\) such that $D_g(b_i)\notin U_s$ or $b_i^{[p]}\notin U_s$, respectively. Adjoin this vector as \(b_{s+1}\). Since \(E\) is finite-dimensional, iteration terminates with a homogeneous basis $\mathcal B_{\mathrm{rt}}=(b_1,\ldots,b_d)$ having properties \textup{(1)} and \textup{(2)}. Tracing the parent of
each \(b_j\notin E_1\) backwards gives a finite routing-compatible ancestry in \(\mathcal B_{\mathrm{rt}}\), proving \textup{(3)}.
\end{proof}

\begin{remark}
\label{conv:general-p-fixed-routing-pbw-basis}
For the remainder of the construction, fix the routing basis
\(\mathcal B_{\mathrm{rt}}=(b_1,\ldots,b_d)\). All restricted PBW
coefficients, occupancy vectors, certified states, and protected exponent
profiles are taken relative to this single basis. No change of complementary
PBW basis is made during the construction.
\end{remark}

\begin{definition}
\label{def:general-p-stored-provenance}
For each routing-basis direction \(b_j\), fix the routing-compatible ancestry supplied by Lemma~\ref{lem:general-p-fixed-routing-basis}. For each PUSH edge $b_j=D_g(b_i)$, record the parent \(b_i\) and the control \(g\). For each POWER edge $b_j=b_i^{[p]}$, record the parent \(b_i\). These fixed parent data are called the \emph{stored provenance} of the routing basis.
\end{definition}

\begin{definition}
\label{def:general-p-ancestry-admissible-state}
A completed certified restricted PBW state
$\nu=(\nu_1,\ldots,\nu_d)$ is called \emph{ancestry-admissible} if it
is taken relative to the fixed routing basis
$\mathcal B_{\mathrm{rt}}=(b_1,\ldots,b_d)$. In this case every
occupied direction \(b_j\), that is every \(b_j\) with \(\nu_j>0\),
carries the stored provenance fixed in
Definition~\ref{def:general-p-stored-provenance}. Thus every completed
certified restricted PBW state relative to
\(\mathcal B_{\mathrm{rt}}\) is ancestry-admissible.
\end{definition}

\begin{definition}
\label{def:general-p-primitive-create-direction}

The primitive layer is $E_1=G'/D_{(3)}$. Since \(E_1\) is spanned by
images of genuine group commutators, choose the basis of \(E_1\) from
such commutator directions. A routing-basis direction \(b_j\in E_1\)
is called a \emph{primitive CREATE direction} if there exist \(a,b\in G\)
such that the image of \([a,b]\) in \(E_1=G'/D_{(3)}\) is \(b_j\).
A \emph{genuine primitive CREATE} in the direction \(b_j\) is the
chronological operation obtained from this actual group commutator
representative.
\end{definition}

We have therefore fixed both the routing basis and the stored provenance
along which its directions are reached. This determines where the later
chronological constructions must route. The next step is to show that the
distinguished contributions produced along these routes survive the
competing terms created by PUSH and restricted PBW normalization.


\subsection{Descendant sectors and protected noncancellation}
\label{subsec:descendant-protection}


Transport acts as a derivation on \(u(E)\), so a PUSH applied to a PBW
monomial produces several competing terms. We now isolate the distinguished
PUSH branch and organize the remaining terms into descendant sectors. The
purpose of this machinery is to show that positive descendant corrections
cannot return to the same protected PBW coordinate, so the distinguished
routing contribution survives with nonzero coefficient.

\begin{proposition}
\label{prop:general-p-push-expansion}

Let $\nu=(\nu_1,\ldots,\nu_d)$ be a completed certified restricted PBW state relative to the fixed routing basis $\mathcal B_{\mathrm{rt}}=(b_1,\ldots,b_d)$, and write
$M(\nu)=b_1^{\nu_1}\cdots b_d^{\nu_d}$. Then for every \(g\in G\),
\[
        D_g(M(\nu))
        =
        \sum_{\substack{1\leq i\leq d\\ \nu_i>0}}
        \nu_i
        b_1^{\nu_1}
        \cdots
        b_i^{\nu_i-1}
        D_g(b_i)
        \cdots
        b_d^{\nu_d}.
\]
Moreover, each summand corresponds to PUSH acting on exactly one
occupied column. For the \(i\)-th summand, one copy of \(b_i\) is
replaced by its PUSH image \(D_g(b_i)\), so that the \(b_i\)-occupancy
changes from \(\nu_i\) to \(\nu_i-1\), while the other occupancies are
unchanged before restricted PBW normalization. Thus the complete PUSH
output is the sum of the contributions obtained by acting on each
occupied column separately.
\end{proposition}

\begin{proof}
Since \(D_g\) is a derivation and \(u(E)\) is commutative,
\[
    D_g(b_i^{\nu_i})
    =
    \nu_i b_i^{\nu_i-1}D_g(b_i).
\]
Applying the Leibniz rule to
$M(\nu)=b_1^{\nu_1}\cdots b_d^{\nu_d}$ therefore gives
\[
    D_g(M(\nu))
    =
    \sum_{\substack{1\leq i\leq d\\ \nu_i>0}}
    \nu_i
    b_1^{\nu_1}\cdots
    b_i^{\nu_i-1}D_g(b_i)
    \cdots b_d^{\nu_d}.
\]
In the \(i\)-th summand, the derivation acts only on the factor
\(b_i^{\nu_i}\), replacing one copy of \(b_i\) by \(D_g(b_i)\) and
leaving all other factors unchanged before restricted PBW
normalization. Hence each summand represents PUSH acting on one
occupied column, and the displayed sum is the complete PUSH output.
\end{proof}

\begin{corollary}
\label{cor:general-p-distinguished-push-branch}
Fix an occupied routing-basis direction \(b_i\), and suppose $D_g(b_i)\neq0$. Then the \(i\)-th term of the PUSH expansion gives the distinguished branch
$b_i^{\nu_i}\rightsquigarrow\nu_i b_i^{\nu_i-1}D_g(b_i)$. Equivalently, up to the nonzero scalar $\nu_i\in\F_p^\times$, one certified copy of \(b_i\) is transported to \(D_g(b_i)\), while the remaining $\nu_i-1$ copies of \(b_i\) remain in place.
\end{corollary}

\begin{proof}
By Proposition~\ref{prop:general-p-push-expansion}, the \(i\)-th
summand is
\[
    \nu_i b_1^{\nu_1}\cdots
    b_i^{\nu_i-1}D_g(b_i)\cdots b_d^{\nu_d},
\]
giving $b_i^{\nu_i} \rightsquigarrow \nu_i b_i^{\nu_i-1}D_g(b_i)$. Since \(1\leq\nu_i\leq p-1\), one has \(\nu_i\in\F_p^\times\); and since \(D_g(b_i)\neq0\), this transports
one copy of \(b_i\) to \(D_g(b_i)\), leaving the remaining \(\nu_i-1\) copies unchanged.
\end{proof}

The desired PUSH contribution is now identified. We next collect the
competing positive PUSH descendants and show that they remain separated
from the direction whose PBW coordinate is being protected.

\begin{definition}
\label{def:general-p-positive-push-descendants}
Define the \emph{positive PUSH descendant space} of \(u\) by
\[
        \mathscr D^+(u)
        :=
        \Span_{\F_p}
        \left\{
        D_{g_r}\cdots D_{g_1}(u)
        :
        r\geq1,\;
        g_1,\ldots,g_r\in G
        \right\},
\]
where zero iterated descendants may of course be omitted. Thus every element used to generate $\mathscr D^+(u)$ has undergone at least one genuine PUSH.
\end{definition}

\begin{definition}
\label{def:general-p-descendant-sector}
The \emph{descendant sector} of \(u\), denoted by \(\mathscr S(u)\),
is the smallest graded \(\F_p\)-subspace of \(E\) containing the
positive PUSH descendant space \(\mathscr D^+(u)\) and closed under
restricted \(p\)-POWER.
\end{definition}

\begin{lemma}
\label{lem:general-p-descendant-sector-closure}
The descendant sector \(\mathscr S(u)\) is closed under homogeneous
decomposition, \(\F_p\)-linear combinations, and restricted
\(p\)-POWER. In particular, if $x\in\mathscr S(u)\cap E_j$, then
$x^{[p]}\in\mathscr S(u)\cap E_{pj}$.
\end{lemma}

\begin{proof}
These properties follow directly from the definition of
\(\mathscr S(u)\) as the graded restricted subspace of \(E\) generated
by \(\mathscr D^+(u)\).
\end{proof}

\begin{lemma}
\label{lem:general-p-descendants-higher-weight}
Every nonzero homogeneous element $v\in\mathscr S(u)$ satisfies $\wt(v)>\wt(u)=m$. In particular, $u\notin\mathscr S(u)$.
\end{lemma}

\begin{proof}
Every positive PUSH descendant has the form $D_{g_r}\cdots D_{g_1}(u)$, $r\geq1$. Since $D_g:E_j\longrightarrow E_{j+1}$, every nonzero such descendant belongs to
$E_{m+r}$ and therefore has weight strictly greater than \(m\). Passing to homogeneous linear combinations does not decrease homogeneous weight. Restricted POWER sends
$E_j\longrightarrow E_{pj}$. If \(j>m\), then $pj>m$. Hence the graded restricted closure of the positive PUSH descendants cannot produce a nonzero homogeneous element of weight at most \(m\).
Therefore every nonzero homogeneous $v\in\mathscr S(u)$ has weight \(>m\). Since $u\in E_m$, it follows that $u\notin\mathscr S(u)$.
\end{proof}

\begin{lemma}
\label{lem:general-p-descendant-sector-push-stable}
Applying a further PUSH to a descendant-sector element produces either
zero or another element of the descendant sector. More precisely, for
every \(g\in G\), $D_g(\mathscr S(u))\subseteq\mathscr S(u)$.
\end{lemma}

\begin{proof}
The space \(\mathscr D^+(u)\) is spanned by iterated positive PUSH descendants $D_{g_r}\cdots D_{g_1}(u)$, $r\geq1$. Applying another \(D_g\) gives either zero or another positive PUSH
descendant. Hence $D_g(\mathscr D^+(u)) \subseteq \mathscr D^+(u)$. It remains to check compatibility with restricted closure. For homogeneous
$x\in\mathscr S(u)$, Lemma~\ref{lem:general-p-transport-power-zero} gives $D_g(x^{[p]})=0$. Transport is linear and respects homogeneous decomposition because
$D_g(E_j)\subseteq E_{j+1}$. Therefore \(D_g\) preserves the smallest graded restricted subspace containing the positive PUSH descendants. Hence
$D_g(\mathscr S(u)) \subseteq \mathscr S(u)$.
\end{proof}

Thus positive descendants remain in a PUSH-stable sector of strictly higher
homogeneous weight. We now convert this weight separation into a device for
protecting a distinguished PBW coordinate through subsequent reconstruction
and normalization.

\begin{definition}
\label{def:general-p-saturated-descendant-factor}

Let \(U\) be the span of the homogeneous directions whose occupancies are protected during the current reconstruction, and put
$S=\mathscr S(u)$. Choose a homogeneous decomposition $S=(S\cap U)\oplus T$ and an ordered homogeneous basis \(b_1,\ldots,b_q\) of \(T\). Define
$\Omega_{S,U}^{(p)} := b_1^{p-1}\cdots b_q^{p-1}$. This is called the \emph{relative descendant separator}. The chosen homogeneous basis of \(T\) is used only to construct this auxiliary
separator; it does not replace the fixed routing basis \(\mathcal B_{\mathrm{rt}}\), relative to which all certified PBW coordinates and protected exponent profiles are computed.
\end{definition}

\begin{definition}
\label{def:general-p-protected-reconstruction-word}
Let \(u\in E\) be homogeneous, let \(W\) be a finite genuine chronological routing word for \(u\), and let \(U\) be the span of the homogeneous directions whose occupancies are to be protected. Put $S=\mathscr S(u)$, choose a homogeneous decomposition $S=(S\cap U)\oplus T$, and let $\Omega_{S,U}^{(p)}$ be the corresponding relative descendant separator from Definition~\ref{def:general-p-saturated-descendant-factor}. We call the resulting data $(W,U,S,\Omega_{S,U}^{(p)})$ a \emph{protected reconstruction datum} for \(u\).
\end{definition}

\begin{lemma}
\label{lem:general-p-occupancy-normalization}
Let \(A=u(E)\). For \(n\geq 0\), let \(F_nA\) be the
\(\mathbb F_p\)-span of all products of at most \(n\) elements of \(E\).
Then
\[
F_0A\subseteq F_1A\subseteq\cdots
\]
is the occupancy filtration of \(A\). Every nontrivial use of a
restricted PBW normalization $x^p=x^{[p]}$ lowers occupancy by exactly \(p-1\). Consequently, complete restricted
PBW normalization is occupancy-nonincreasing, and is strictly
occupancy-decreasing whenever at least one restricted relation is used.
\end{lemma}

\begin{proof}
A product of \(r\) elements of \(E\) has occupancy \(r\). Before normalization, the factor \(x^p\) has occupancy \(p\). Since $x^{[p]}\in E$, after normalization the same factor has occupancy \(1\). Thus one use of $x^p=x^{[p]}$ lowers occupancy by $p-1$. Every further restricted normalization has the same effect. Hence complete restricted PBW normalization never increases occupancy, and it decreases occupancy strictly whenever at least one restricted
relation is used.
\end{proof}

\begin{lemma}
\label{lem:general-p-relative-descendant-separation}
Let $(W,U,S,\Omega_{S,U}^{(p)})$ be a protected reconstruction datum for \(u\), where \(U\) is the span of the protected routing-basis directions and $S=(S\cap U)\oplus T$.
Let $\Omega_{S,U}^{(p)}=c_1^{p-1}\cdots c_q^{p-1}$, where \(c_1,\ldots,c_q\) is the chosen homogeneous basis of \(T\). Consider a distinguished top-occupancy PBW contribution and an
unwanted contribution of the same pre-normalization occupancy in which the distinguished direction is replaced by a nonzero homogeneous positive PUSH descendant correction $s=s_U+t$,
$s_U\in S\cap U$, $t\in T$.. After multiplication by \(\Omega_{S,U}^{(p)}\) and complete restricted PBW normalization, the unwanted contribution either has a different
protected exponent profile or lies in strictly lower occupancy. Consequently it cannot contribute to the distinguished protected top-occupancy PBW coordinate.
\end{lemma}

\begin{proof}
Write $s=s_U+t$, $s_U\in S\cap U$, $t\in T$. Suppose first that \(t=0\). Then $0\neq s=s_U\in S\cap U\subseteq U$. Since \(U\) is spanned by the protected routing-basis directions,
expansion of \(s_U\) in the fixed routing basis \(\mathcal B_{\mathrm{rt}}\) contains a nonzero protected direction. Thus, before any restricted normalization, every resulting nonzero
term differs from the distinguished contribution in at least one protected exponent. If no restricted normalization occurs, this different protected exponent profile remains visible in the final restricted PBW coordinate. If a restricted normalization does occur, then Lemma~\ref{lem:general-p-occupancy-normalization} gives a strict occupancy drop. Since the distinguished contribution is being taken at the top-occupancy PBW coordinate, such a term cannot return to that coordinate. Now suppose that \(t\neq0\). Write $t=\lambda_1c_1+\cdots+\lambda_qc_q$, with some \(\lambda_j\neq0\). Then
\[
t\Omega_{S,U}^{(p)}
=
\sum_{j=1}^q
\lambda_j
c_1^{p-1}\cdots c_j^p\cdots c_q^{p-1}.
\]
For every \(j\) with \(\lambda_j\neq0\), the corresponding summand contains the factor \(c_j^p\). Hence the restricted relation $c_j^p=c_j^{[p]}$ must be used. By Lemma~\ref{lem:general-p-occupancy-normalization}, that summand loses at least \(p-1\) units of occupancy under complete restricted PBW normalization. The unwanted contribution had the same pre-normalization occupancy as the distinguished contribution. Therefore, after this forced normalization, every contribution arising from the nonzero \(T\)-component has strictly smaller occupancy than the distinguished top-occupancy contribution. Complete restricted PBW normalization cannot increase occupancy, so no such contribution can return to the distinguished top-occupancy PBW coordinate. Thus the \(S\cap U\)-component is separated either by a different protected exponent profile or by strict occupancy loss, while every nonzero \(T\)-component is separated by forced strict occupancy loss. Hence the unwanted positive PUSH descendant correction cannot contribute to the distinguished protected top-occupancy PBW coordinate.
\end{proof}

\begin{lemma}
\label{lem:general-p-normalization-lowers-occupancy}
Every nontrivial use of a restricted PBW normalization
$x^p=x^{[p]}$ lowers total occupancy by exactly \(p-1\).
Consequently, if a chronological contribution undergoes at least one
nontrivial restricted normalization, its final PBW occupancy is strictly
smaller than its pre-normalization occupancy.
\end{lemma}

\begin{proof}
Before normalization, the displayed factor $x^p$ has occupancy \(p\).
After normalization, $x^{[p]}\in E$ is one homogeneous PBW direction and
therefore has occupancy \(1\). Hence the occupancy drop is \(p-1\).
If further restricted normalizations are required, each lowers occupancy
by another \(p-1\). Thus complete restricted PBW normalization is
occupancy-nonincreasing and is strictly occupancy-decreasing whenever at
least one restricted relation is used.
\end{proof}

\begin{lemma}
\label{lem:general-p-protected-factor-through}
Let \(W\) be a finite genuine chronological routing word for \(u\), and
let \(U\) be the span of the homogeneous directions whose occupancies
are protected along \(W\). At the distinguished top-occupancy PBW
coordinate, multiplication by the relative descendant separator
\(\Omega_{\mathscr S(u),U}^{(p)}\) separates every unwanted positive
descendant correction. After complete restricted PBW normalization,
each such correction either has a different protected exponent profile
or lies in strictly lower occupancy. Hence it cannot contribute to the
distinguished protected PBW coordinate.
\end{lemma}

\begin{proof}
Put $S:=\mathscr S(u)$, $S=(S\cap U)\oplus T$, and let \(b_1,\ldots,b_q\) be the homogeneous basis of \(T\) defining
$\Omega_{S,U}^{(p)} = b_1^{p-1}\cdots b_q^{p-1}$. Consider an unwanted positive descendant correction $s=s_U+t$, $s_U\in S\cap U$,
$t\in T$. Each branch arising from \(s\) replaces the distinguished contribution in an otherwise identical chronological term, and therefore has the
same pre-normalization occupancy as the corresponding distinguished branch. If \(t=0\), then \(s=s_U\neq0\) lies in the protected span \(U\), so
every nonzero routing-basis component of \(s\) changes an actual protected exponent. Hence, unless restricted normalization occurs, its protected \(p\)-PBW profile differs from that of the distinguished coordinate. If normalization occurs, Lemma~\ref{lem:general-p-normalization-lowers-occupancy} gives a strict occupancy drop, so the branch cannot return to the
distinguished top-occupancy coordinate. Now suppose \(t\neq0\). Write
\[
t=\sum_{j=1}^q\alpha_jb_j.
\]
Then
\[
t\Omega_{S,U}^{(p)}
=
\sum_{j=1}^q
\alpha_j
b_1^{p-1}\cdots b_j^p\cdots b_q^{p-1}.
\]
Every nonzero summand therefore undergoes $b_j^p=b_j^{[p]}$, and so loses \(p-1\) units of occupancy. Complete restricted PBW normalization cannot increase occupancy; hence no such summand can
reach the distinguished separator-multiplied top-occupancy coordinate. Thus every unwanted correction is separated either by its protected exponent profile or by strict occupancy loss. Equivalently, no normalized branch arising from an unwanted positive descendant correction can equal the distinguished protected PBW monomial, and therefore no such correction can contribute to its PBW coefficient. Applying this at each of the finitely many letters of \(W\) proves the result.
\end{proof}

\begin{corollary}
\label{cor:general-p-lower-occupancy-cannot-return}
Let \(B\) be a distinguished PBW branch of occupancy \(N\). If a competing branch has entered occupancy \(<N\) through restricted normalization, then no subsequent restricted PBW normalization can
return that branch to the PBW coordinate of \(B\). In particular, restricted normalization cannot raise a lower-occupancy competitor back to a top-occupancy coordinate.
\end{corollary}

\begin{proof}
Every subsequent restricted normalization is occupancy-nonincreasing, and every nontrivial normalization lowers occupancy by \(p-1\). Thus a branch which has reached occupancy \(<N\) remains at occupancy strictly less than \(N\) throughout the remainder of PBW normalization. Since equality of restricted PBW monomials implies equality of their occupancy degrees, such a branch cannot normalize to the distinguished occupancy-\(N\) monomial.
\end{proof}

\begin{lemma}
\label{lem:general-p-protected-noncancellation}
Suppose a finite genuine chronological construction produces a
distinguished restricted PBW branch \(B\) with coefficient
\(\lambda\in\F_p^\times\). Assume that every competing chronological
contribution, after being followed to complete restricted PBW normal
form, satisfies at least one of the following:
\begin{enumerate}
\item no nontrivial restricted normalization occurs, and its final
restricted PBW exponent vector differs from that of \(B\);
\item at least one nontrivial restricted normalization occurs, and
its final occupancy is strictly smaller than the occupancy of \(B\).
\end{enumerate}
Then the coefficient of \(B\) in the complete PBW-normalized output
is \(\lambda\neq0\).
\end{lemma}

\begin{proof}
Cancellation at \(B\) can occur only from another chronological
contribution whose complete restricted PBW normal form contains
exactly the same restricted PBW monomial \(B\). A contribution of
type \emph{(1)} has a different final PBW exponent vector and
therefore cannot equal \(B\). A contribution of type \emph{(2)}
has strictly smaller occupancy and, by
Corollary~\ref{cor:general-p-lower-occupancy-cannot-return},
cannot normalize to \(B\). Hence no competing contribution reaches
the distinguished PBW coordinate, and its coefficient remains
\(\lambda\in\F_p^\times\).
\end{proof}

\begin{corollary}
\label{cor:general-p-protected-push}
Let $\mathbf C=\bigl((u_1,\varepsilon_1),\ldots,(u_r,\varepsilon_r)\bigr)$ be a completed certified restricted PBW state, and fix an occupied index \(i\). Suppose $z:=D_g(u_i)\neq0$ and that this PUSH is protected relative to the occupancies fixed in the current state. Then the complete genuine chronological PUSH output contains the distinguished branch
$u_i^{\varepsilon_i}\rightsquigarrow\varepsilon_i u_i^{\varepsilon_i-1}z$
with nonzero coefficient.
\end{corollary}

\begin{proof}
The PUSH expansion gives the distinguished contribution
\[
    \varepsilon_i
    u_1^{\varepsilon_1}\cdots
    u_i^{\varepsilon_i-1}z
    \cdots
    u_r^{\varepsilon_r}.
\]
Since $1\leq\varepsilon_i\leq p-1$, its coefficient
$\varepsilon_i\in\F_p^\times$ is nonzero. Let \(U\) be the protected span and put $S=\mathscr S(u_i)$. By Lemma~\ref{lem:general-p-protected-factor-through}, every competing contribution is separated from the distinguished PBW coordinate either by a different protected exponent profile or by strict occupancy loss after restricted normalization. Hence Lemma~\ref{lem:general-p-protected-noncancellation} applies, and the distinguished PUSH coordinate survives with nonzero coefficient.
\end{proof}

The protection mechanism is now in place: distinguished PUSH contributions
can be separated from positive descendants and survive complete restricted
PBW normalization with nonzero coefficient. We can therefore turn from the
routing framework to the concrete chronological operations that will use it.

\subsection{CREATE, FILL, and POWER primitives}
\label{subsec:create-fill-power-primitives}


We now introduce the primitive operations used by the chronological
construction. CREATE supplies the first certified copy of a primitive
direction, FILL increases the occupancy of an already occupied direction,
and POWER realizes a stored restricted POWER edge. Together with protected
PUSH, these operations provide the local moves needed for reconstruction.

\begin{definition}
\label{def:general-p-primitive-create-available}
A primitive CREATE direction \(b_j\) is \emph{available} in the
completed certified state \(\nu=(\nu_1,\ldots,\nu_d)\) if \(\nu_j=0\).
Thus \(b_j\) is available precisely when its current occupancy is zero,
so that a primitive CREATE introduces the first certified copy of
\(b_j\), changing its occupancy from $0\longrightarrow 1$.
\end{definition}

\begin{lemma}
\label{lem:general-p-carrier-bracket}

Let \(C=hY\in \F_pG\), where \(h,g\in G\) and \(Y\in\F_pG\), and put
\[
    c=[h,g]_{\mathrm{grp}}
      =h^{-1}g^{-1}hg.
\]
Then
\[
    [C,g]_{\Lie}
    =
    hg\Bigl((g^{-1}Yg-Y)+(1-c^{-1})Y\Bigr).
\]
Consequently, if the leading graded state of \(Y\) is \(X\in u(E)\), then the next homogeneous graded component of the single genuine continuation
$C\longmapsto [C,g]_{\Lie}$ is the sum of $D_g(X)$ and the primitive multiplication contribution $\epsilon\,Xb$, where \(b\in E_1\) is the image of \(c\) and
\(\epsilon\in\F_p^\times\) depends only on the chosen group-commutator convention. Thus both the PUSH terms and the primitive CREATE term occur as
branches of one and the same appended genuine Lie bracket.
\end{lemma}

\begin{proof}
Since $hg=ghc$, we have $gh=hg\,c^{-1}$. Hence
\[
\begin{aligned}
    [hY,g]_{\Lie}
      &=hYg-ghY\\
      &=hg(g^{-1}Yg)-hg\,c^{-1}Y\\
      &=hg\Bigl((g^{-1}Yg-Y)+(1-c^{-1})Y\Bigr).
\end{aligned}
\]
This is an identity in \(\F_pG\). Passing to the next homogeneous component, the first summand gives the transport contribution \(D_g(X)\).  Since \(c\in G'\) and its image in \(G'/D_{(3)}=E_1\) is \(b\), the leading component of \(1-c^{-1}\) is \(\epsilon b\) for some \(\epsilon\in\F_p^\times\).  Therefore the second summand has distinguished graded contribution
$\epsilon Xb$. Both contributions arise from the single actual bracket \([C,g]_{\Lie}\), so no formal graded multiplication is being identified with an independent chronological operation.
\end{proof}

\begin{lemma}
\label{lem:general-p-primitive-create}

Let \(b_j\in E_1\) be a primitive CREATE direction, and suppose that $C=hY$ is the current actual chronological representative and that
\(g\in G\) is chosen so that the image of $[h,g]_{\mathrm{grp}}$ in \(E_1=G'/D_{(3)}\) is \(b_j\). If \(X\in u(E)\) is the current leading graded state, then the single
genuine continuation $C\longmapsto [C,g]_{\Lie}$ has, in its next homogeneous graded component, a distinguished CREATE contribution $\epsilon\,Xb_j$, $\epsilon\in\F_p^\times$. Thus, at the level of restricted PBW bookkeeping, the distinguished CREATE branch acts, up to a nonzero scalar, by $X\longmapsto Xb_j$. In particular, $M(\nu)\longmapsto \epsilon\,M(\nu)b_j$,
and if \(\nu_j=0\), this distinguished contribution changes the
\(b_j\)-occupancy from \(0\) to \(1\).
\end{lemma}

\begin{proof}
By Lemma~\ref{lem:general-p-carrier-bracket}, the next homogeneous graded component of the actual continuation $C\longmapsto [C,g]_{\Lie}$ is the sum of the transport contribution and the primitive multiplication contribution $\epsilon\,Xb$, $\epsilon\in\F_p^\times$, where \(b\in E_1\) is the image of \([h,g]_{\mathrm{grp}}\). By hypothesis this image is \(b_j\). Hence the distinguished CREATE branch is $\epsilon\,Xb_j$. Since \(\epsilon\neq0\), its restricted PBW coefficient is nonzero precisely when the corresponding coefficient of \(Xb_j\) is nonzero.
Therefore the distinguished PBW contribution of \(M(\nu)\) is $\epsilon\,M(\nu)b_j$. If \(\nu_j=0\), no restricted normalization is required at the \(b_j\)-coordinate, and its exponent in this contribution is exactly \(1\).
\end{proof}

\begin{lemma}
\label{lem:general-p-carrier-freedom}
Let \(C\in \F_pG\) have leading graded state \(X\in u(E)\), and let \(a\in G\).  Put $Y:=a^{-1}C$, so that $C=aY$. Then \(Y\) has the same positive-degree leading graded state \(X\).
Consequently, for any primitive direction \(b_j\in E_1\) and any \(a,b\in G\) satisfying $[a,b]_{\mathrm{grp}}D_{(3)}=b_j$, the current chronological representative may be written as
\(C=aY\) and Lemma~\ref{lem:general-p-carrier-bracket} may be applied with \(h=a\) and \(g=b\).
\end{lemma}

\begin{proof}
Since \(a^{-1}\in G\), $a^{-1}=1+(a^{-1}-1)$, where \(a^{-1}-1\) has positive filtration degree.  Thus $a^{-1}C = C+(a^{-1}-1)C$, and the second summand lies in strictly higher filtration degree than the leading homogeneous component of \(C\).  Hence left multiplication by \(a^{-1}\) does not alter the leading graded state.  Therefore \(Y=a^{-1}C\) has the same leading graded state \(X\).
\end{proof}

\begin{theorem}
\label{thm:general-p-protected-create}
Let $\nu=(\nu_1,\ldots,\nu_d)$ be a completed certified \(p\)-PBW state relative to the fixed routing basis $\mathcal B_{\mathrm{rt}}=(b_1,\ldots,b_d)$, and let \(b_j\in E_1\) be an available primitive CREATE direction, so that \(\nu_j=0\).  Then there exists one genuine chronological Lie-bracket continuation producing a complete graded output in which the state
$\nu'=(\nu_1,\ldots,\nu_{j-1},1,\nu_{j+1},\ldots,\nu_d)$ is completed certified.  In particular, $\nu_j'=1$, $\nu_i'=\nu_i\quad(i\neq j)$, so CREATE realizes the elementary occupancy transition
$0\longrightarrow1$ while preserving every previously certified occupancy.
\end{theorem}

\begin{proof}
Let \(C\) be the current actual left-normed commutator whose complete leading graded output is \(X\), and suppose that $[M(\nu)]_{\mathcal B_{\mathrm{rt}}}X =\lambda \in\F_p^\times$.
Since \(b_j\) is a primitive CREATE direction, choose \(a,b\in G\) such that $[a,b]_{\mathrm{grp}}D_{(3)}=b_j$. By Lemma~\ref{lem:general-p-carrier-freedom}, write
$C=aY$ with \(Y=a^{-1}C\), where \(Y\) has the same leading graded state \(X\). Append the single genuine Lie bracket with \(b\): $C\longmapsto [C,b]_{\Lie}$. By Lemma~\ref{lem:general-p-carrier-bracket}, its next homogeneous graded component is the sum of the transport contribution $D_b(X)$ and a primitive CREATE contribution
$\epsilon Xb_j$, $\epsilon\in\F_p^\times$. We first compute the coefficient of $M(\nu)b_j$ in the CREATE contribution.  Let $N=b_1^{\alpha_1}\cdots b_d^{\alpha_d}$
be a restricted PBW monomial occurring in \(X\). If \(\alpha_j\leq p-2\), then multiplication by \(b_j\) requires no restricted normalization at the \(j\)-th coordinate, and
$Nb_j=M(\nu)b_j \, \Longrightarrow\, N=M(\nu)$. If \(\alpha_j=p-1\), then $Nb_j$ contains \(b_j^p\), and restricted normalization gives $b_j^p=b_j^{[p]}\in E_p$. After normalization the \(b_j\)-exponent is therefore \(0\), so such a term cannot equal \(M(\nu)b_j\), whose \(b_j\)-exponent is \(1\). Hence
\[
    [M(\nu)b_j]_{\mathcal B_{\mathrm{rt}}}
    (\epsilon Xb_j)
    =
    \epsilon
    [M(\nu)]_{\mathcal B_{\mathrm{rt}}}X
    =
    \epsilon\lambda
    \neq0.
\]

It remains to check that the simultaneous transport contribution \(D_b(X)\) cannot cancel this coordinate.  Every summand of \(D_b(X)\) is obtained by transporting one already displayed PBW
factor and therefore replaces one existing factor by its positive PUSH descendant; it does not adjoin an additional primitive \(b_j\)-factor.  Thus before restricted normalization it has no larger
total PBW occupancy than the originating monomial.  Any nontrivial restricted normalization lowers occupancy by \(p-1\).  By contrast, the distinguished CREATE coordinate
$M(\nu)b_j$ has total occupancy $|\nu|_{\mathrm{occ}}+1$. Therefore no transport contribution can normalize to this PBW coordinate. Consequently
$[M(\nu)b_j]_{\mathcal B_{\mathrm{rt}}}X' = \epsilon\lambda \neq0$, where \(X'\) is the complete graded output of the appended genuine bracket.  Since \(\nu_j=0\), the monomial \(M(\nu)b_j\) is already restricted and has exponent vector $(\nu_1,\ldots,\nu_{j-1},1,\nu_{j+1},\ldots,\nu_d)$. Thus this state is completed certified, every old occupancy is
unchanged, and the \(b_j\)-occupancy has increased from \(0\) to \(1\).
\end{proof}

\begin{corollary} 
\label{cor:general-p-create-deficit} 
Every protected CREATE satisfies $\Delta_p(\mathbf C') \leq \Delta_p(\mathbf C)-1$. If CREATE produces no additional certified occupancy beyond the new 
primitive copy, then $\Delta_p(\mathbf C') = \Delta_p(\mathbf C)-1$. 
\end{corollary} 
 
\begin{proof} 
The new direction acquires occupancy \(1\), while no previously certified occupancy decreases.  Hence $|\mathbf C'|_{\mathrm{occ}} \geq |\mathbf C|_{\mathrm{occ}}+1$. 
The assertion follows from $\Delta_p(\mathbf C)=(p-1)d-|\mathbf C|_{\mathrm{occ}}$. 
\end{proof} 

Protected CREATE handles the case in which a primitive direction has zero
occupancy. Once a direction is already occupied, the corresponding local
operation is FILL, whose purpose is to add one further certified copy without
losing the copies already present.

\begin{definition}
\label{def:general-p-fill}
Suppose the fixed routing-basis direction \(b_i\) has occupancy \(\nu_i=r\), \(1\leq r\leq p-2\). A \emph{FILL at \(b_i\)} is the construction of one further copy of \(b_i\), with distinguished branch $b_i^r\rightsquigarrow b_i^{r+1}$. The purpose of FILL is to increase the occupancy of an already occupied direction by one while retaining the copies already certified.
\end{definition}

FILL increases an existing occupancy until the restricted threshold
\(p-1\) is reached. At that threshold, a stored POWER edge is realized by
constructing one further transient copy and applying restricted
normalization.

\begin{definition}
\label{def:general-p-power}
Suppose $b_j=b_i^{[p]}$ is a stored POWER edge of the fixed routing
basis and $\nu_i=p-1$. A \emph{POWER construction at \(b_i\)} is the
transition $b_i^{p-1}\rightsquigarrow b_i^p \longrightarrow b_i^{[p]}=b_j$,
where the first step consists of constructing one further transient
copy of \(b_i\), and the second is restricted normalization.
\end{definition}

\begin{definition}
\label{def:general-p-power}
Suppose $b_j=b_i^{[p]}$ is a stored POWER edge of the fixed routing
basis and $\nu_i=p-1$. A \emph{POWER construction at \(b_i\)} is the
transition $b_i^{p-1}\rightsquigarrow b_i^p \longrightarrow b_i^{[p]}=b_j$,
where the first step consists of constructing one further transient
copy of \(b_i\), and the second is restricted normalization.
\end{definition}

As with PUSH, it remains to verify that the distinguished POWER contribution
cannot be cancelled by competing reconstruction terms. The same protection
principle, together with the strict weight increase of positive descendants,
provides the required separation.

\begin{lemma}
\label{lem:general-p-power-noncancellation}
Let \(b_j=b_i^{[p]}\) be a stored POWER edge, with \(b_i\in E_m\), and let $M_0=b_i^{p-1}M$ be a certified restricted PBW monomial.  Suppose a protected genuine reconstruction of one further copy of \(b_i\) has distinguished transient contribution $\lambda b_i^pM$, $\lambda\in\F_p^\times$. Assume that all previously certified directions other than \(b_i\) are protected during the reconstruction. Then, after complete restricted PBW normalization, the contribution $\lambda b_i^pM \longmapsto \lambda b_jM$ cannot be cancelled by an unwanted positive-descendant correction,
provided the exponent of \(b_j\) in \(M\) is at most \(p-2\). Consequently \(b_jM\) occurs with nonzero coefficient in the complete
normalized output.
\end{lemma}

\begin{proof}
The distinguished transient branch is characterized by the fact that the \(p-1\) certified copies of \(b_i\) are joined by one further copy of the same homogeneous direction \(b_i\).  Hence its transient \(b_i\)-exponent is exactly \(p\), and normalization gives $b_i^pM=b_i^{[p]}M=b_jM$. Consider an unwanted reconstruction contribution.  By protected factor-through, it is separated from the distinguished reconstruction branch either by a different protected exponent profile or by the replacement of the distinguished \(b_i\)-copy by a nonzero positive PUSH-descendant correction. In the latter case every homogeneous component of the correction has weight strictly greater than \(m\).  It therefore has no component in the \(E_m\)-direction \(b_i\).  Thus the \(p-1\) previously certified copies of \(b_i\) are not completed to a \(p\)-fold block in such a branch.  Its \(b_i\)-exponent remains \(p-1\), whereas the distinguished branch contains \(b_i^p\) and, after normalization, has \(b_i\)-exponent \(0\).

If a competing contribution undergoes restricted normalization in some other direction, Lemma~\ref{lem:general-p-protected-factor-through} and the protected exponent profile separate it from the distinguished POWER coordinate; every previously certified direction other than \(b_i\) remains protected.  A normalization in a positive descendant direction cannot create \(b_i\), since restricted POWER strictly multiplies homogeneous weight by \(p\). Finally, since the exponent of \(b_j\) in \(M\) is at most \(p-2\), the distinguished replacement $b_i^pM\longmapsto b_jM$
requires no further normalization at \(b_j\). Hence no competing normalized branch reaches the same restricted PBW coordinate \(b_jM\).  Its coefficient is therefore \(\lambda\neq0\).
\end{proof}

We have now established the fixed routing ancestry, the protection and
noncancellation machinery, and the primitive CREATE, FILL, PUSH, and POWER
operations. These are precisely the ingredients required for the
reconstruction step. Reconstruction, Certified FILL, and Certified POWER
depend on one another only at strictly smaller homogeneous weights, so they
will be proved together in the next section by simultaneous strong induction.


\section{Simultaneous Reconstruction, FILL, and POWER}
\label{sec:simultaneous-reconstruction-fill-power}

The routing and protection machinery of the preceding section is now used to construct routing-basis directions by genuine chronological operations. The central point is that reconstruction, FILL, and POWER cannot be treated independently: reconstruction along a stored POWER edge requires FILL and POWER at a lower-weight parent, while FILL and POWER themselves require the reconstruction of further distinguished occurrences along stored ancestry. We therefore establish the three operations simultaneously by strong induction on homogeneous weight. The resulting theorem is the main recursive engine for the remainder of the argument. Reconstruction follows the stored provenance of a routing-basis direction, while FILL and POWER allow distinguished restricted PBW occupancies to be increased without losing the coordinates already protected. The strict decrease of homogeneous weight along every recursive ancestry call is what makes the simultaneous construction well founded.

\begin{theorem}
\label{thm:general-p-simultaneous-reconstruction-fill-power}
Fix the routing basis
\(\mathcal B_{\mathrm{rt}}=(b_1,\ldots,b_d)\) with its stored provenance.
For every homogeneous weight \(m\), the following three assertions hold
simultaneously.

\begin{enumerate}

\item[\textup{(R)}]
If \(b_i\in E_m\), then there exists a finite genuine chronological reconstruction word \(W_{b_i}\), following the stored ancestry of \(b_i\), which produces one distinguished occurrence of \(b_i\) with nonzero coefficient. Moreover, \(\wt_{\mathrm{route}}(W_{b_i})\) equals the number of genuine Lie-bracket increments in its chronological realization.

\item[\textup{(F)}] Let \(\nu=(\nu_1,\ldots,\nu_d)\) be an ancestry-admissible completed certified restricted PBW state occurring in a genuine chronological history. If \(b_i\in E_m\) and $1\leq \nu_i=r\leq p-2$, then there exists a finite genuine protected chronological continuation to a completed certified state \(\nu'=(\nu_1',\ldots,\nu_d')\) such that
$\nu_i'\geq r+1$, $\nu_j'\geq\nu_j$ $(1\leq j\leq d)$. The distinguished FILL branch realizes \(r\longrightarrow r+1\).

\item[\textup{(P)}] Let \(\nu=(\nu_1,\ldots,\nu_d)\) be an ancestry-admissible completed certified restricted PBW state occurring in a genuine chronological history. Suppose
$b_j=b_i^{[p]}$, $b_i\in E_m$, $\nu_i=p-1$, $\nu_j\leq p-2$, where \(b_i\mapsto b_j\) is a stored POWER edge. Then there exists a finite genuine protected chronological continuation whose distinguished branch produces $b_i^{p-1}\rightsquigarrow b_i^p=b_i^{[p]}=b_j$, with a protected nonzero restricted PBW coordinate in the direction \(b_j\). All previously certified occupancies outside the consumed \(b_i^{p-1}\)-block remain protected.

\end{enumerate}

These assertions hold by simultaneous strong induction on \(m\):
every recursive reconstruction, FILL, or POWER construction used in their
proof occurs at strictly smaller homogeneous weight. Consequently there is
no same-weight circular dependence among \textup{(R)}, \textup{(F)}, and
\textup{(P)}.
\end{theorem}

The proof follows the stored ancestry recursively. Primitive CREATE directions provide the weight-one base case. A stored PUSH edge raises homogeneous weight by one, whereas a stored POWER edge passes from a parent of weight \(q\) to a child of weight \(pq\). Thus every reconstruction needed at a proper ancestry parent lies at strictly smaller weight, which is precisely the decrease required for the simultaneous induction.

\begin{proof}
We prove \textup{(R)}, \textup{(F)}, and \textup{(P)} simultaneously by strong induction on the homogeneous weight \(m\). For each \(m\geq1\), assume that all three assertions have already been established for every homogeneous weight strictly smaller than \(m\). We prove the three assertions at weight \(m\).

\medskip
\noindent
\textbf{Proof of \textup{(R)}.} Let \(b_i\in E_m\) be equipped with its stored routing-compatible ancestry. If \(m=1\), then by Lemma~\ref{lem:general-p-fixed-routing-basis} and
Definition~\ref{def:general-p-primitive-create-direction}, \(b_i\) is a primitive CREATE direction arising from a genuine group commutator. The corresponding primitive CREATE therefore gives a finite genuine chronological reconstruction word \(W_{b_i}\) producing \(b_i\) with nonzero coefficient. Its routing weight is exactly the number of genuine Lie-bracket increments in its chronological realization. Now let \(m>1\). By Definition~\ref{def:general-p-stored-provenance}, the final ancestry edge of \(b_i\) is either a PUSH edge or a POWER edge. Suppose first that
$a\xmapsto{\mathrm{PUSH}}b_i=D_g(a)$. Then \(a\in E_{m-1}\), so \(\wt(a)=m-1<m\). By the induction hypothesis \textup{(R)}, \(a\) admits a finite genuine chronological reconstruction word
\(W_a\) with nonzero distinguished coefficient. Append the stored PUSH. By Corollary~\ref{cor:general-p-protected-push}, the distinguished contribution \(D_g(a)=b_i\) survives with nonzero coefficient. Hence \(W_a\) followed by this PUSH is the required reconstruction word \(W_{b_i}\). Since the PUSH contributes one genuine Lie-bracket increment and has routing-weight increment
\(1\), the required routing-weight equality follows from the induction hypothesis.

Suppose instead that $a\xmapsto{\mathrm{POWER}}b_i=a^{[p]}$, $a\in E_q$. Then \(m=pq\), and hence \(q<m\). By the induction hypothesis \textup{(R)}, reconstruct one distinguished occurrence of \(a\). By the lower-weight induction hypothesis \textup{(F)}, apply the required FILL steps, if any, to reach distinguished occupancy \(p-1\). These constructions are all available at weight \(q<m\). Applying the terminal POWER then produces the transient contribution \(a^p\), which normalizes to $a^p=a^{[p]}=b_i$. Lemma~\ref{lem:general-p-power-noncancellation} ensures that this distinguished normalized contribution survives with nonzero coefficient. Thus the reconstruction of \(a\), followed by the required FILL steps and terminal POWER, gives the required finite genuine chronological reconstruction word \(W_{b_i}\). Finally, restricted normalization \(a^p=a^{[p]}\) contributes no genuine Lie-bracket increment. Therefore, using the induction hypotheses for the constituent chronological constructions, \(\wt_{\mathrm{route}}(W_{b_i})\) is exactly the number of genuine Lie-bracket increments in its chronological realization. This proves
\textup{(R)}.

\medskip
\noindent
\textbf{Proof of \textup{(F)}.}
Let \(\nu=(\nu_1,\ldots,\nu_d)\) be an ancestry-admissible completed certified restricted PBW state occurring in a genuine chronological history, and suppose that \(b_i\in E_m\) has \(\nu_i=r\), where \(1\leq r\leq p-2\). Write $M(\nu)=b_i^rM$. Follow the stored ancestry of \(b_i\). Every proper ancestry parent has strictly smaller homogeneous weight: a backward PUSH step decreases the
weight from \(s+1\) to \(s\), while a backward POWER step decreases it from \(ps\) to \(s\). Hence all reconstruction, FILL, and POWER constructions required along the proper ancestry are available from the strong induction hypothesis. Using these lower-weight constructions, reconstruct one further distinguished occurrence of \(b_i\) while protecting the previously certified occupancies. At PUSH edges use Corollary~\ref{cor:general-p-protected-push}, and at POWER edges use the corresponding lower-weight FILL and POWER constructions. Thus we obtain a finite genuine protected chronological continuation whose distinguished branch adjoins one further copy of \(b_i\) with nonzero coefficient.

Consequently, if the certified coefficient of \(M(\nu)\) is \(\lambda\in\F_p^\times\) and the reconstructed occurrence has coefficient \(\mu\in\F_p^\times\), the distinguished contribution is
$\lambda\mu\,b_i^{r+1}M$. Since \(r\leq p-2\), we have \(r+1\leq p-1\), so no restricted normalization occurs at the distinguished \(b_i\)-factor. By Lemma~\ref{lem:general-p-protected-factor-through} and Corollary~\ref{cor:general-p-lower-occupancy-cannot-return}, every unwanted descendant contribution remains separated from this distinguished protected PBW coordinate. Lemma~\ref{lem:general-p-protected-noncancellation} therefore shows that its coefficient remains nonzero after complete restricted PBW normalization. Hence the continuation determines a completed certified state \(\nu'=(\nu_1',\ldots,\nu_d')\) satisfying $\nu_i'\geq r+1$, $\nu_j'\geq\nu_j$ $(1\leq j\leq d)$. Thus the distinguished FILL branch realizes
\(r\longrightarrow r+1\) while preserving all previously certified occupancies. This proves \textup{(F)} at weight \(m\).

\medskip
\noindent
\textbf{Proof of \textup{(P)}.}
Let \(\nu=(\nu_1,\ldots,\nu_d)\) be as in \textup{(P)}, with \(b_j=b_i^{[p]}\), \(b_i\in E_m\), \(\nu_i=p-1\), and \(\nu_j\leq p-2\). Write $M(\nu)=b_i^{p-1}M$. Follow the stored ancestry of \(b_i\). Every proper ancestry parent has strictly smaller homogeneous weight, so the required reconstruction, FILL, and POWER constructions along the proper ancestry are supplied by the strong
induction hypothesis. Using these lower-weight constructions, and Corollary~\ref{cor:general-p-protected-push} at PUSH edges, reconstruct one further distinguished occurrence of \(b_i\) while protecting all previously certified occupancies outside the distinguished \(b_i^{p-1}\)-block. Thus the resulting finite genuine protected continuation contains a distinguished transient contribution $\lambda\mu\,b_i^pM$, $\lambda,\mu\in\F_p^\times$. Restricted normalization gives $b_i^pM=b_i^{[p]}M=b_jM$. Since the exponent of \(b_j\) in \(M\) is at most \(p-2\),
Lemma~\ref{lem:general-p-power-noncancellation} applies and shows that this distinguished normalized coordinate survives with nonzero coefficient. Moreover, the new \(b_j\)-exponent is at most \(p-1\), so no further restricted normalization occurs at \(b_j\). Hence the complete output contains a protected nonzero restricted PBW coordinate in the direction \(b_j\), while all previously certified occupancies outside the consumed \(b_i^{p-1}\)-block remain protected. This proves \textup{(P)} at weight \(m\).

\medskip
Finally, along every stored ancestry a backward PUSH step strictly decreases
weight, and a backward POWER step replaces weight \(ps\) by \(s<ps\).
Therefore every recursive reconstruction, FILL, or POWER invocation above is
at strictly smaller homogeneous weight. The strong induction consequently
closes, with no same-weight circular dependence among \textup{(R)},
\textup{(F)}, and \textup{(P)}.
\end{proof}

We record separately the three consequences of the simultaneous theorem that will be used later. The first is the reconstruction statement itself: every stored routing-basis direction can be produced by a genuine chronological word following its stored ancestry. The next two isolate the certified FILL and POWER operations, which will be invoked when modifying completed restricted PBW states.

\begin{corollary}
\label{lem:general-p-ancestry-reconstruction}
Let \(0\neq u\in E_m\) be a routing-basis direction equipped with its
stored routing-compatible ancestry. Then there exists a finite genuine
chronological reconstruction word \(W_u\) whose distinguished branch
produces \(u\) with nonzero coefficient. Moreover,
\(\wt_{\mathrm{route}}(W_u)\) equals the number of genuine Lie-bracket
increments in its chronological realization. The reconstruction is recursive along the stored ancestry. A terminal
PUSH edge is realized by reconstruction of its parent followed by the
stored PUSH, while a terminal POWER edge is realized by reconstruction
of its parent, the required FILL steps, and POWER at that parent.
\end{corollary}

\begin{proof}
Apply Theorem~\ref{thm:general-p-simultaneous-reconstruction-fill-power}, assertion \textup{(R)}, to the routing-basis direction \(u\in E_m\) with its stored routing-compatible ancestry. This gives a finite genuine chronological reconstruction word \(W_u\) whose distinguished branch produces \(u\) with nonzero coefficient, and for which \(\wt_{\mathrm{route}}(W_u)\) equals the number of genuine Lie-bracket increments in its chronological realization.

The recursive description is part of the same assertion. If the terminal ancestry edge is a PUSH edge \(a\xmapsto{\mathrm{PUSH}}u=D_g(a)\), then \(W_u\) is obtained from a reconstruction of \(a\) followed by the stored PUSH. If the terminal ancestry edge is a POWER edge \(a\xmapsto{\mathrm{POWER}}u=a^{[p]}\), then \(W_u\) is obtained from a reconstruction of \(a\), followed by the required FILL steps and POWER at \(a\). This proves the corollary.
\end{proof}

The preceding corollary supplies reconstruction as a reusable chronological operation. We next record the corresponding certified-state operation below the restricted boundary: whenever the distinguished occupancy of a routing direction lies between \(1\) and \(p-2\), one further reconstructed occurrence can be adjoined while preserving the previously certified occupancies.

\begin{corollary}
\label{thm:general-p-certified-fill}
Let \(\nu=(\nu_1,\ldots,\nu_d)\) be an ancestry-admissible completed
certified restricted PBW state relative to the fixed routing basis
\(\mathcal B_{\mathrm{rt}}=(b_1,\ldots,b_d)\), and suppose
\(b_i\in E_m\) satisfies \(\nu_i=r\), where \(1\leq r\leq p-2\).
Then there exists a finite genuine protected chronological continuation
to a completed certified state \(\nu'=(\nu_1',\ldots,\nu_d')\) such that
$\nu_i'\geq r+1$, $\nu_j'\geq\nu_j$ $(1\leq j\leq d)$. In particular, FILL realizes \(r\longrightarrow r+1\), and no restricted
normalization occurs at the distinguished \(b_i\)-factor.
\end{corollary}

\begin{proof}
Apply Theorem~\ref{thm:general-p-simultaneous-reconstruction-fill-power}, assertion \textup{(F)}, to the ancestry-admissible completed certified state \(\nu\) and the direction \(b_i\in E_m\) with \(\nu_i=r\), where \(1\leq r\leq p-2\). This gives a finite genuine protected chronological continuation to a completed certified state \(\nu'=(\nu_1',\ldots,\nu_d')\) satisfying \(\nu_i'\geq r+1\) and \(\nu_j'\geq\nu_j\) for \(1\leq j\leq d\). The distinguished FILL branch adjoins one further reconstructed occurrence of \(b_i\) to the existing \(b_i^r\)-factor and therefore realizes \(r\longrightarrow r+1\). Since \(r+1\leq p-1\), no restricted normalization occurs at the distinguished \(b_i\)-factor. This proves the corollary.
\end{proof}

FILL handles all occupancy increases which remain strictly below the restricted boundary. When the distinguished occupancy has reached \(p-1\), the next reconstructed occurrence instead creates a \(p\)-fold block, and restricted normalization converts that block into the stored POWER child. The following corollary records this boundary operation.

\begin{corollary}
\label{thm:general-p-certified-power}
Let \(\nu=(\nu_1,\ldots,\nu_d)\) be an ancestry-admissible completed
certified restricted PBW state relative to the fixed routing basis
\(\mathcal B_{\mathrm{rt}}=(b_1,\ldots,b_d)\), and suppose
\(b_j=b_i^{[p]}\) is a stored POWER edge with \(b_i\in E_m\),
\(\nu_i=p-1\), and \(\nu_j\leq p-2\). Then there exists a finite genuine
protected chronological continuation producing the distinguished
transient branch $b_i^{p-1}\rightsquigarrow b_i^p$. Restricted normalization gives
$b_i^p=b_i^{[p]}=b_j$, and the resulting complete chronological output contains the
corresponding protected nonzero restricted PBW coordinate in the
direction \(b_j\). All previously certified occupancies not consumed
in the distinguished \(b_i^{p-1}\)-block remain protected throughout
the continuation.
\end{corollary}

\begin{proof}
Apply Theorem~\ref{thm:general-p-simultaneous-reconstruction-fill-power}, assertion \textup{(P)}, to the ancestry-admissible completed certified state \(\nu\) and the stored POWER edge \(b_j=b_i^{[p]}\), with \(b_i\in E_m\), \(\nu_i=p-1\), and \(\nu_j\leq p-2\). This gives a finite genuine protected chronological continuation in which one further distinguished reconstructed occurrence of \(b_i\) produces the transient branch \(b_i^{p-1}\rightsquigarrow b_i^p\). Restricted normalization gives \(b_i^p=b_i^{[p]}=b_j\), and assertion \textup{(P)} guarantees that the resulting complete chronological output contains the corresponding protected nonzero restricted PBW coordinate in the direction \(b_j\). Since \(\nu_j\leq p-2\), adjoining the new \(b_j\)-factor does not cross the restricted boundary at \(b_j\), and all previously certified occupancies not consumed in the distinguished \(b_i^{p-1}\)-block remain protected throughout the continuation. This proves the corollary.
\end{proof}

We have therefore obtained the three chronological operations needed for the subsequent construction. Stored ancestry gives a genuine reconstruction of every routing-basis direction; Certified FILL increases a distinguished occupancy while it remains below the restricted boundary; and Certified POWER crosses that boundary by converting a completed \(p\)-fold block into its stored restricted-power direction. In each case the protected PBW machinery prevents the distinguished contribution from being lost to competing descendant terms. These operations will now be used as the local moves in the BUILD REBUILD procedure. The purpose of that procedure is to organize repeated applications of reconstruction, FILL, and POWER so that the required routing-basis occupancies can be assembled while preserving those already certified.

\section{BUILD REBUILD}
\label{sec:build-rebuild}

The simultaneous reconstruction theorem of the preceding section supplies the local chronological operations needed to move through the fixed routing basis. We now organize those operations into the two complementary mechanisms BUILD and REBUILD. BUILD creates new certified occupancy by FILL, PUSH, and POWER along stored ancestry. Such growth may temporarily spend previously certified source occupancy, so REBUILD records the resulting losses and restores them recursively. The central issue is therefore not merely to produce a new routing-basis direction, but to obtain positive net progress without losing what has already been certified. Stored provenance provides the routes along which both growth and restoration take place, while backward descent in homogeneous weight supplies the well-foundedness needed for the restoration process. The section culminates in a completed BUILD REBUILD move which retains the new occupancy produced by BUILD and restores every previously certified occupancy spent in producing it.

\subsection{Ancestry and BUILD operations}
\label{subsec:ancestry-build}

We begin by recording the ancestry structure needed to organize BUILD. Every nonprimitive routing-basis direction has a stored parent, reached either through a PUSH edge or through a POWER edge. Tracing these edges backwards strictly lowers homogeneous weight and therefore eventually reaches the primitive layer. Forward along the same ancestry, the certified FILL, PUSH, and POWER operations provide the elementary moves from which BUILD is assembled.

\begin{definition}
\label{def:general-p-parent-child}
Let $b_i \longrightarrow b_j$ be a stored provenance edge. The direction \(b_i\) is called the \emph{parent} and \(b_j\) the \emph{child}. More precisely, if $b_j=D_g(b_i)$,
then \(b_j\) is a \emph{PUSH child} of \(b_i\), while if $b_j=b_i^{[p]}$, then \(b_j\) is a \emph{POWER child} of \(b_i\).
\end{definition}

\begin{lemma}
\label{lem:general-p-backward-weight-descent}
Every nonprimitive backward step along stored provenance strictly decreases homogeneous weight. More precisely,
$E_{m+1}\text{-PUSH child} \rightsquigarrow E_m\text{-parent}$, and
$E_{pm}\text{-POWER child} \rightsquigarrow E_m\text{-parent}$.
\end{lemma}

\begin{proof}
If $b=D_g(a)$ and $a\in E_m$, then $b\in E_{m+1}$. Hence $\wt(a)=m<m+1=\wt(b)$. If $b=a^{[p]}$ and $a\in E_m$, then $b\in E_{pm}$. Since
$p\geq2$ and $m\geq1$, we have $m<pm$. Thus every backward provenance step strictly lowers positive
homogeneous weight.
\end{proof}

\begin{corollary}
\label{cor:general-p-finite-backward-provenance}
There is no infinite backward chain of stored provenance. Every backward provenance chain terminates at a primitive direction in \(E_1\).
\end{corollary}

Thus every stored ancestry is finite in the backward direction. We now use the chronological operations established in the preceding section in the forward direction. FILL supplies ordinary occupancy growth below the restricted boundary, POWER crosses that boundary when a stored POWER child is required, and PUSH transports an occupied parent into its stored PUSH child.

\begin{proof}
Homogeneous weights are positive integers, and every backward step strictly decreases weight by Lemma~\ref{lem:general-p-backward-weight-descent}.
\end{proof}

\begin{corollary}
\label{cor:general-p-fill-progress}
Every Certified FILL satisfies $\Delta_p(\mathbf C')  \leq  \Delta_p(\mathbf C)-1$. If the protected reconstruction creates no additional certified occupancy beyond the distinguished FILL, then
$\Delta_p(\mathbf C') = \Delta_p(\mathbf C)-1$.
\end{corollary}

\begin{proof}
The distinguished occupancy of \(u\) increases by one: $r\longrightarrow r+1$, while every previously certified multiplicity is preserved. Hence $|\mathbf C'|_{\mathrm{occ}}
\geq |\mathbf C|_{\mathrm{occ}}+1$. Since $\Delta_p(\mathbf C) = (p-1)d-|\mathbf C|_{\mathrm{occ}}$, we obtain
$\Delta_p(\mathbf C')  \leq  \Delta_p(\mathbf C)-1$.
\end{proof}

\begin{corollary}
\label{cor:general-p-successive-fill}
Suppose $\varepsilon_{\mathbf C}(u)=r$ with $1\leq r<p-1$. Then successive Certified FILL operations give a finite protected continuation
\[
r \longrightarrow r+1 \longrightarrow \cdots \longrightarrow p-1
\]
while preserving all previously certified multiplicities. No restricted normalization occurs during any of these steps.
\end{corollary}

\begin{proof}
Apply Theorem~\ref{thm:general-p-certified-fill} successively at occupancies $r,r+1,\ldots,p-2$. At the \(s\)-th stage, $s+1\leq p-1$, so no restricted relation is triggered. There are exactly
$p-1-r$ such FILL steps.
\end{proof}

Successive FILL therefore brings an occupied direction to maximal ordinary restricted PBW occupancy without triggering normalization. Once occupancy \(p-1\) has been reached, the next distinguished occurrence is handled differently: POWER forms the \(p\)-fold block and restricted normalization converts it into the stored POWER child.

\begin{corollary}
\label{cor:general-p-power-routing-direction}
Under the hypotheses of Theorem~\ref{thm:general-p-certified-power}, put $z:=b_i^{[p]}=b_j$. Then the complete chronological output contains a protected nonzero
\(z\)-coordinate. If $z\notin \Span_{\F_p}\mathcal O(\mathbf C)$, then POWER has produced access to a new homogeneous direction outside
the previously occupied frame.
\end{corollary}

\begin{proof}
By Theorem~\ref{thm:general-p-certified-power}, the POWER construction produces a surviving protected nonzero coordinate in the direction $b_i^{[p]}=b_j=z$. Thus the direction \(z\) is present in the complete chronological output. If $z\notin \Span_{\F_p}\mathcal O(\mathbf C)$, then \(z\) was not contained in the span of the previously occupied directions. Hence the surviving \(z\)-coordinate gives access to a new independent homogeneous direction outside the previously occupied frame.
\end{proof}

\begin{proposition}
\label{prop:general-p-fill-before-power}
Suppose a stored routing-compatible ancestry contains a POWER edge $u\longmapsto z:=u^{[p]}$, and the current completed certified state satisfies $\varepsilon_{\mathbf C}(u)=r$,
$1\le r<p-1$, and $\varepsilon_{\mathbf C}(z)\le p-2$. Then the POWER edge is reconstructed in two stages. First perform successive Certified FILL operations
\[
r\longrightarrow r+1\longrightarrow\cdots\longrightarrow p-1.
\]
After maximal ordinary PBW occupancy has been reached, perform
Certified POWER:
\[
u^{p-1}\rightsquigarrow u^p=u^{[p]}=z.
\]
Every previously certified multiplicity other than the \(u\)-block
consumed by the POWER normalization remains protected throughout.
\end{proposition}

\begin{proof}
By Theorem~\ref{thm:general-p-certified-fill}, successive protected
FILL operations give
\[
r\longrightarrow r+1\longrightarrow\cdots\longrightarrow p-1,
\]
while preserving all previously certified multiplicities. Once $\varepsilon(u)=p-1$, the hypotheses of Theorem~\ref{thm:general-p-certified-power} are satisfied. Hence Certified POWER gives the genuine protected transition $u^{p-1}\rightsquigarrow u^p=u^{[p]}=z$. The \(p\)-fold \(u\)-block is consumed by this restricted normalization, while every other previously certified multiplicity
remains protected. Thus the stored POWER edge is crossed only after its parent has first been filled to maximal ordinary PBW occupancy.
\end{proof}

The preceding proposition describes the complete BUILD move along a stored POWER edge: first fill the parent to maximal ordinary occupancy and then cross the restricted boundary. We now record the corresponding move along a stored PUSH edge. Here one certified source copy is transported to the child, so the new child occupancy is obtained at the possible cost of one unit of source occupancy. This temporary loss is the phenomenon that will lead to REBUILD.

\begin{definition}
\label{def:general-p-distinguished-push}
Suppose $b_j=D_g(b_i)$ is a stored PUSH edge of the fixed routing basis and \(\nu_i>0\). The \emph{distinguished PUSH branch at \(b_i\)} is the Leibniz summand obtained by applying \(D_g\) to exactly one of the \(\nu_i\) displayed copies of \(b_i\). Thus $b_i^{\nu_i} \rightsquigarrow \nu_i b_i^{\nu_i-1}b_j$.
\end{definition}

\begin{theorem}
\label{thm:general-p-protected-push}
Let $\nu=(\nu_1,\ldots,\nu_d)$ be a completed certified \(p\)-PBW state and suppose $b_j=D_g(b_i)$ is a stored PUSH edge with \(\nu_i>0\) and \(\nu_j\leq p-2\). Then the complete genuine chronological PUSH output contains the protected distinguished branch $b_i^{\nu_i} \rightsquigarrow \nu_i b_i^{\nu_i-1}b_j$. Its coefficient is nonzero: $\nu_i\in\F_p^\times$. The distinguished contribution is already in restricted PBW form at the \(b_j\)-coordinate and cannot disappear by cancellation.
\end{theorem}

\begin{proof}
Write $M(\nu)=b_i^{\nu_i}M$, $[M(\nu)]_{\mathcal B_{\mathrm{rt}}}X =\lambda\in\F_p^\times$. Realize the stored PUSH edge by the corresponding genuine chronological continuation. Its PUSH component acts on the leading graded state by the derivation \(D_g\). Since $D_g(b_i)=b_j$, we have
\[
D_g\left(b_i^{\nu_i}M\right)
=
D_g\left(b_i^{\nu_i}\right)M
+
b_i^{\nu_i}D_g(M)
=
\nu_i b_i^{\nu_i-1}b_jM
+
b_i^{\nu_i}D_g(M).
\]
Hence the distinguished PUSH contribution is $\lambda\nu_i b_i^{\nu_i-1}b_jM$. Since $1\leq\nu_i\leq p-1$, its coefficient satisfies $\lambda\nu_i\in\F_p^\times$. Since \(\nu_j\leq p-2\), the distinguished PUSH increases the \(b_j\)-occupancy by one to at most \(p-1\). Therefore no restricted normalization occurs at the \(b_j\)-coordinate on the distinguished branch. Let \(U\) be the protected span, with the one copy transported from \(b_i\) excluded from the protected occupancy requirement, and put $S=\mathscr S(b_i)$. The remaining Leibniz summands, together with the additional carrier contribution arising in the same genuine chronological bracket, are regarded as competing contributions.

Use the relative descendant separator only as an auxiliary
coefficient-detection device. By
Lemma~\ref{lem:general-p-protected-factor-through}, every competing
contribution is separated from the distinguished restricted PBW
coordinate either by a different protected exponent profile or by
strict occupancy loss after restricted normalization. Hence
Lemma~\ref{lem:general-p-protected-noncancellation} applies, and the
distinguished PUSH coordinate survives with nonzero coefficient. Thus the complete genuine chronological output contains the protected
distinguished branch $b_i^{\nu_i} \rightsquigarrow \nu_i b_i^{\nu_i-1}b_j$, and this branch cannot disappear by cancellation.
\end{proof}


\subsection{REBUILD complexity and termination}
\label{subsec:rebuild-termination}

BUILD can create genuine new occupancy, but a PUSH may spend a certified copy of its source and a POWER construction may consume certified copies of its parent. We now formalize these temporary losses as occupancy debts and show that they can always be repaid. The guiding principle is to reconstruct a missing direction by following its stored provenance backwards. Any new restoration requirements created in doing so occur at strictly smaller homogeneous weight. We encode the outstanding requirements as a finite multiset of weights and use the standard multiset ordering as a well-founded complexity measure for REBUILD.

\begin{definition}
\label{def:general-p-unit-occupancy-debt}
Let $\nu=(\nu_1,\ldots,\nu_d)$ be the completed certified state
immediately before a protected routing operation. If a previously
certified routing-basis direction $b_i\in E_m$ had occupancy
$\nu_i=r$, but the protected post-routing state contains only \(r-1\)
certified copies of \(b_i\), then the missing copy is called a
\emph{unit occupancy debt at \(b_i\)}. Its weight is
$\wt(b_i)=m$. Thus protected PUSH creates exactly one unit occupancy
debt at its source whenever the spent source copy has not already been
restored.
\end{definition}

\begin{definition}
\label{def:general-p-occupancy-debt}
Let $\nu=(\nu_1,\ldots,\nu_d)$ be the completed certified state before
a protected growth construction, and let
$\nu^{\mathrm{cur}}=(\nu_1^{\mathrm{cur}},\ldots,\nu_d^{\mathrm{cur}})$
record the currently protected occupancies relative to the fixed routing
basis $\mathcal B_{\mathrm{rt}}=(b_1,\ldots,b_d)$. For every previously
certified routing-basis direction \(b_i\), define $\delta_i :=  \nu_i-\nu_i^{\mathrm{cur}}$.
Whenever \(\delta_i>0\), this number is the \emph{occupancy debt at
\(b_i\)}. Each missing copy is regarded as one separate unit debt.
\end{definition}

\begin{definition}
\label{def:general-p-debt-multiset}
At any stage of the construction, let $\mathcal D$ be the finite multiset
containing one copy of $\wt(b_i)$ for each outstanding unit occupancy
debt at a routing-basis direction \(b_i\). Thus if \(b_i\in E_m\) has
occupancy debt \(\delta_i=q>0\), then \(m\) occurs in \(\mathcal D\)
with multiplicity \(q\).
\end{definition}

\begin{definition}
\label{def:general-p-debt-ordering}

Order finite multisets of positive integers by the standard multiset extension $<_{\mathrm{mul}}$ of the usual strict order $<$ on $\mathbb Z_{>0}$.Thus replacing one entry \(n\) by 
finitely many entries all strictly smaller than \(n\) strictly decreases the multiset.
\end{definition}

The point of this ordering is that REBUILD follows stored ancestry backwards. Repaying a debt of a given weight may create further restoration requirements, but backward weight descent ensures that every such new requirement has strictly smaller weight. Thus the debt multiset records not only how much occupancy remains to be restored, but also a well-founded measure of the complexity of that restoration.

\begin{lemma}
\label{lem:general-p-debt-well-founded}
The ordering $<_{\mathrm{mul}}$ on finite debt multisets is well founded.
\end{lemma}

\begin{proof}
The usual order on $\mathbb Z_{>0}$ is well founded. The standard multiset extension of a well-founded strict order is again well founded.
Hence there is no infinite strictly descending chain $\mathcal D_0  >  \mathcal D_1  >  \mathcal D_2  >  \cdots$.
\end{proof}
The preceding definitions record the bookkeeping cost of protected growth. We now turn this bookkeeping into the restoration process REBUILD. During REBUILD, outstanding debts are repaid through stored ancestry, but the new occupancy obtained by the BUILD step must remain protected throughout the repair. Likewise, every old occupancy unit which is not currently recorded as debt must remain certified. The protected REBUILD invariant records precisely this data. It separates occupancy which must be preserved at the present stage from occupancy which has been deliberately spent and is therefore represented by an outstanding debt.

\begin{definition}
\label{def:general-p-rebuild-invariant}

During REBUILD let $\mathscr P$ consist of

\begin{enumerate}
\item every newly protected direction or multiplicity produced by the
growth step which initiated REBUILD;
\item every previously certified occupancy unit which is not currently
represented by an outstanding unit occupancy debt.
\end{enumerate}

Thus a previously certified occupancy unit belongs to $\mathscr P$
exactly while it is currently certified. If such a unit is deliberately
spent during a PUSH or POWER reconstruction, it is removed from
$\mathscr P$ and immediately recorded as an outstanding unit occupancy
debt. When that debt is repaid, the unit re-enters $\mathscr P$. The \emph{protected REBUILD invariant} is the requirement that every
intermediate chronological state continue to contain a completed
certificate for all data currently belonging to $\mathscr P$.
\end{definition}

We can now describe the recursive step of REBUILD. To repay a selected unit debt, follow the final edge of its stored provenance. If that edge is PUSH, restoring the child may transfer the debt to its lower-weight parent. If it is POWER, restoring the child may replace the selected debt by several debts at the lower-weight POWER parent. In either case the selected debt is removed and any replacement debts have strictly smaller homogeneous weight.

\begin{lemma}
\label{lem:general-p-push-debt-reduction}
Suppose there is an outstanding unit debt at \(b\in E_{m+1}\), with
stored PUSH provenance $a\longmapsto b=D_g(a),\qquad a\in E_m$.
Then protected PUSH restores the missing copy of \(b\).  In doing so it
spends one copy of \(a\); if that copy was previously protected, the
debt at \(b\) is replaced by a unit debt at \(a\).  Thus a debt of
weight \(m+1\) is either removed entirely or replaced by a debt of
strictly smaller weight \(m\).
\end{lemma}

\begin{proof}
Since \(b\) has a unit debt, its current occupancy satisfies $\nu_b^{\mathrm{cur}}\leq p-2$. Hence the child-room hypothesis of Theorem~\ref{thm:general-p-protected-push} holds.  Protected PUSH along the stored edge \(a\mapsto b\) has surviving distinguished branch $a^r\rightsquigarrow r a^{r-1}b$, and therefore restores the missing copy of \(b\). If the spent copy of \(a\) was previously protected, its loss creates one unit debt at \(a\); otherwise no new debt is created.  Since
\[
    \wt(a)=m<m+1=\wt(b),
\]
the resulting debt problem has strictly smaller weight.
\end{proof}

\begin{lemma}
\label{lem:general-p-power-debt-reduction}
Suppose there is an outstanding unit debt at \(b\in E_{pm}\), with
stored POWER provenance $a\longmapsto b=a^{[p]},\qquad a\in E_m$.
Then protected POWER restores the missing copy of \(b\).  Any previously
protected copies of \(a\) consumed in doing so are recorded as unit
debts at \(a\).  Thus the debt of weight \(pm\) is removed and replaced
by at most \(p-1\) debts, each of strictly smaller weight \(m\).
\end{lemma}

\begin{proof}
Since \(b\) has a unit debt, its current occupancy satisfies $\nu_b^{\mathrm{cur}}\leq p-2$,
so the child-room hypothesis for Certified POWER holds. Let \(r\) be the current protected occupancy of \(a\).  Reconstruct
\(a\), if necessary, and apply successive Certified FILL operations until its occupancy is \(p-1\).  Certified POWER then gives
$a^{p-1}\rightsquigarrow a^p=a^{[p]}=b$, so the missing copy of \(b\) is restored. The POWER normalization may consume previously protected copies of
\(a\).  Each such copy is recorded as one unit debt at \(a\); copies introduced only for the POWER construction create no debt.  Hence at most \(p-1\) new debts are created, all of weight
\[
\wt(a)=m<pm=\wt(b).
\]
All protected data outside the consumed \(a\)-block remain protected.
Thus repayment of the POWER-child debt replaces it only by strictly
lower-weight parent debts.
\end{proof}

\begin{lemma}
\label{lem:general-p-rebuild-complexity-decrease}
Each completed REBUILD repayment of a selected outstanding unit debt
strictly decreases the outstanding debt multiset $\mathcal D$ in the
multiset ordering.
\end{lemma}
\begin{proof}
Choose an outstanding unit debt of weight \(n\), and complete its REBUILD repayment. If its last provenance edge is PUSH, then \(n=m+1\), and Lemma~\ref{lem:general-p-push-debt-reduction} repays the selected debt. Any new occupancy debts created in doing so have weight at most \(m<n\). If its last provenance edge is POWER, then \(n=pm\). Complete the lower-weight reconstruction and FILL requirements needed to perform the POWER repayment. The selected debt of weight \(pm\) is then repaid. Every previously protected parent occupancy unit consumed in this repayment gives rise to a unit debt of weight \(m<pm\); any further occupancy debts arising recursively during the lower-weight reconstruction likewise have weight strictly less than \(pm\). Finally, if \(n=1\), there is no earlier provenance edge. Primitive CREATE repays the selected debt directly, so the corresponding entry \(1\) is simply removed. Thus, after the selected debt-repayment step has been completed, the entry \(n\) has been removed and any new outstanding debt entries created in its repayment have weights strictly smaller than \(n\). Consequently $\mathcal D_{\mathrm{new}} <_{\mathrm{mul}}
    \mathcal D_{\mathrm{old}}$.
\end{proof}

\begin{lemma}
\label{lem:general-p-rebuild-termination}
Every REBUILD recursion terminates after finitely many recursive
debt-repayment steps.
\end{lemma}

\begin{proof}
By Lemma~\ref{lem:general-p-rebuild-complexity-decrease}, each completed
repayment of a selected outstanding unit debt strictly decreases the
finite multiset of outstanding routing-basis occupancy debts in the
well-founded multiset ordering. An infinite REBUILD sequence would
therefore give an infinite strictly descending multiset chain,
contradicting Lemma~\ref{lem:general-p-debt-well-founded}. Hence REBUILD
terminates after finitely many completed debt-repayment steps.
\end{proof}

We have therefore established the two facts needed for REBUILD. Each repayment step preserves the data required by the protected REBUILD invariant, and each completed repayment strictly decreases a well-founded debt measure. Consequently every finite collection of occupancy debts created by BUILD can be eliminated after finitely many recursive restoration steps. The remaining task is to package this terminating restoration process with the initial growth operation. This gives the completed BUILD REBUILD move: BUILD creates new protected occupancy, and REBUILD restores any old occupancy spent in obtaining it.


\subsection{BUILD REBUILD reconstruction}
\label{subsec:build-rebuild-reconstruction}

We now combine the local debt-repayment mechanism into a complete chronological restoration theorem. The first result shows that REBUILD restores all occupancy spent by a preceding protected growth construction while retaining the progress produced by that construction. The BUILD REBUILD theorem then packages the initial growth and the finite restoration into a single move with positive net occupancy progress.

\begin{theorem}
\label{thm:general-p-rebuild}
Let $\mathbf C$ be a completed certified \(p\)-PBW state. Suppose a finite protected PUSH or POWER growth construction produces new protected progress and yields a post-growth state
$\mathbf C_1$, but temporarily spends one or more previously certified copies. Then there exists a finite genuine chronological continuation
$\mathbf C_1 \rightsquigarrow \mathbf C'$
such that

\begin{enumerate}
\item every previously certified multiplicity is restored: $\varepsilon_{\mathbf C'}(v) \geq \varepsilon_{\mathbf C}(v)$ for every direction \(v\) certified in \(\mathbf C\);
\item every new protected direction or multiplicity created by the preceding growth step remains certified;
\item every intermediate reconstruction stage satisfies the protected REBUILD invariant.
\end{enumerate}
\end{theorem}

\begin{proof}
Form the finite multiset $\mathcal D$ of outstanding unit occupancy debts and argue by well-founded induction on $\mathcal D$. If $\mathcal D=\varnothing$, there is nothing to reconstruct.
Otherwise choose an outstanding unit debt at $b\in E_n$. If $n=1$, protected primitive CREATE removes the debt while preserving every datum in the protected REBUILD invariant. If $n>1$, follow the last stored provenance edge of \(b\). For a PUSH edg $b=D_g(a)$, Lemma~\ref{lem:general-p-push-debt-reduction} reconstructs the required copy of \(b\), while preserving every datum in the protected REBUILD invariant. If the reconstruction spends a previously protected source copy of \(a\), it creates only subordinate debt of strictly smaller weight. For a POWER edge
$b=a^{[p]}$, Lemma~\ref{lem:general-p-power-debt-reduction} reconstructs the required copy of \(b\) by lower-weight reconstruction, successive FILL, and POWER, while preserving every datum in the protected REBUILD invariant. Any subordinate occupancy debts created in this reconstruction have strictly smaller weight. Thus every recursive step preserves the protected REBUILD invariant
and strictly decreases $\mathcal D$ in the multiset ordering. By Lemma~\ref{lem:general-p-rebuild-termination}, the process terminates after finitely many steps with $\mathcal D=\varnothing$.
Hence every multiplicity certified in \(\mathbf C\) is restored, while every new protected direction or multiplicity produced before REBUILD remains certified throughout. Therefore the resulting state \(\mathbf C'\) satisfies $\varepsilon_{\mathbf C'}(v) \geq \varepsilon_{\mathbf C}(v)$ for every direction \(v\) certified in \(\mathbf C\), and every intermediate reconstruction stage satisfies the protected REBUILD invariant.
\end{proof}

\begin{theorem}
\label{thm:general-p-build-rebuild-reconstruction}
Let $\nu=(\nu_1,\ldots,\nu_d)$ be a completed certified \(p\)-PBW state. Suppose a protected initial growth operation of type PUSH or POWER produces a new certified occupancy unit, possibly while spending previously certified occupancy. Then there exists a finite genuine chronological continuation to a completed certified state $\nu'=(\nu_1',\ldots,\nu_d')$ such that
$\nu_i'\geq\nu_i$ $(1\leq i\leq d)$, while the new occupancy produced by the initial growth operation remains certified. Consequently the completed BUILD REBUILD move has strictly positive net occupancy progress whenever its initial growth step reaches an incompletely certified target occupancy.
\end{theorem}

\begin{proof}
Perform the protected growth step. Record every old occupancy unit thereby spent as a unit occupancy debt. Apply Theorem~\ref{thm:general-p-rebuild}. It repays every such debt in finitely many steps while retaining the new protected occupancy created by the initial growth step. Hence the final completed state contains every previously certified occupancy together with the new protected occupancy.
\end{proof}

\begin{lemma}
\label{lem:general-p-ancestry-admissibility-preserved}
Let \(\nu\) be an ancestry-admissible completed certified state. Every completed certified state obtained from \(\nu\) by one of the protected constructions CREATE, FILL, PUSH, POWER, or REBUILD is again ancestry-admissible.
\end{lemma}

\begin{proof}
FILL changes only the occupancy of an already certified routing-basis direction, so its stored provenance is unchanged. CREATE introduces a primitive routing-basis direction in \(E_1\), whose stored ancestry has length zero. PUSH and POWER follow stored edges of the fixed routing basis, so their target directions already carry the provenance fixed in Definition~\ref{def:general-p-stored-provenance}. Finally, REBUILD restores previously certified routing-basis directions using their already stored provenance. Hence every resulting completed certified state is ancestry-admissible.
\end{proof}

The BUILD REBUILD mechanism is now complete. A protected PUSH or POWER growth step can create new certified occupancy even when doing so temporarily spends old occupancy; REBUILD records those losses as debts and restores them in finitely many steps, without sacrificing the new protected progress. The completed move therefore preserves every previously certified occupancy while retaining the new occupancy created by BUILD. Moreover, ancestry-admissibility is preserved throughout, so the same mechanism remains available for subsequent growth steps. This is the local progress mechanism needed for PBW reachability. In the next section it will be applied to a nonterminal certified state: FILL handles a deficient direction which is already occupied, while an unoccupied direction is reached through its stored ancestry by the appropriate CREATE, PUSH, or POWER construction, with BUILD REBUILD restoring any occupancy spent along the way.


\section{PBW Reachability}
\label{sec:pbw-reachability}

The preceding section established the local BUILD REBUILD mechanism:
from protected growth along the fixed routing basis, any previously
certified occupancy spent in producing new progress can be restored
while the new occupancy remains certified. We now apply that mechanism
to an arbitrary nonterminal completed certified state.

The argument is an exhaustion of the possible deficient occupancies.
An occupied deficient direction is handled by FILL, an unoccupied
primitive direction by CREATE, and an unoccupied nonprimitive direction
is reached along its stored ancestry by PUSH or POWER. In the latter
case BUILD REBUILD restores any occupancy spent in producing the new
direction. Thus in every case all previous certified occupancy is
preserved while the total certified occupancy strictly increases.


\begin{theorem}
\label{thm:pbw-reachability}
There exists a fixed ordered homogeneous basis \(\mathcal C=(c_1,\ldots,c_d)\) of \(E\) such that, whenever \(X\in u(E)\) is the complete graded output of one genuine chronological
routing history and \(\nu\) is a nonterminal restricted PBW state completed certified in \(X\) relative to \(\mathcal C\), there exists a finite genuine chronological continuation from \(X\) to a complete graded output \(X'\) in which a restricted PBW state \(\nu'\) is completed certified, every previously certified occupancy is preserved, and $\Delta_p(\nu')\leq\Delta_p(\nu)-1$.
\end{theorem}

\begin{proof}
\label{proof:pbw-reachability}
Since \(\nu\) is nonterminal, choose \(b_j\) with \(\nu_j<p-1\). If
\(\nu_j>0\), Corollary~\ref{thm:general-p-certified-fill} gives the
required continuation. If \(\nu_j=0\) and \(b_j\in E_1\), protected
CREATE does so. Suppose therefore that \(\nu_j=0\) and
\(b_j\notin E_1\). Write the stored ancestry of \(b_j\) as
\[
    y_0\longmapsto y_1\longmapsto\cdots\longmapsto y_s=b_j,
    \qquad y_0\in E_1.
\]
If \(y_0\) is unoccupied, protected CREATE produces one certified copy
of \(y_0\), and hence gives the required occupancy progress. Otherwise,
since \(y_0\) is occupied and \(y_s=b_j\) is unoccupied, let \(k\geq1\)
be least such that \(y_k\) is unoccupied. Then \(y_{k-1}\) is occupied.
Put \(a:=y_{k-1}\) and \(b:=y_k\). If
\(a\xmapsto{\mathrm{PUSH}}b\), protected PUSH produces one certified
copy of \(b\); if \(a\xmapsto{\mathrm{POWER}}b\),
Corollary~\ref{cor:general-p-successive-fill} followed by
Corollary~\ref{thm:general-p-certified-power} does so. In either case
Theorem~\ref{thm:general-p-build-rebuild-reconstruction} restores any
previously certified occupancy spent in producing \(b\), while retaining
the new copy. Thus every old occupancy is preserved and at least one
deficient coordinate gains one certified unit. Hence
$|\nu'|_{\mathrm{occ}}\geq|\nu|_{\mathrm{occ}}+1$, and therefore
$\Delta_p(\nu')\leq\Delta_p(\nu)-1$. Ancestry-admissibility is preserved
by Lemma~\ref{lem:general-p-ancestry-admissibility-preserved}. Taking
\(\mathcal C=\mathcal B_{\mathrm{rt}}\) proves
Theorem~\ref{thm:pbw-reachability}.
\end{proof}

Thus every nonterminal completed certified state admits a finite genuine
chronological continuation which preserves all existing certified
occupancy and decreases the occupancy deficit by at least one. Since
the deficit is a nonnegative integer, repeated application of the
reachability theorem gives the finite progression toward terminal
certification required in the next section.

\section{Terminal Completion and Equality}
\label{sec:terminal-completion-equality}

The reachability theorem gives a local step: from every nonterminal completed certified state, one can make finite genuine chronological progress while preserving all previously certified occupancy. We now iterate that step to reach the terminal restricted PBW state and then convert the resulting terminal certificate into the required comparison of the lower and upper Lie nilpotency indices. There are three stages. First, the successive reachability continuations are concatenated into one genuine chronological history and iterated until the occupancy deficit is zero. Second, the resulting terminal monomial in the routing basis is compared with the terminal PBW monomial fixed earlier. This produces the required lower bound for the lower Lie nilpotency index over \(\F_p\). Finally, the reverse inequality follows from the formal containment of lower Lie powers in upper Lie powers, and scalar extension carries the resulting equality from \(\F_pG\) to \(KG\). Before iterating reachability, we record the elementary chronological fact that a finite sequence of genuine continuations can be regarded as one genuine routing history. This ensures that iteration does not leave the single-history framework used throughout the construction.

\begin{lemma}
\label{lem:general-p-concatenation-genuine-continuations}
Suppose $X_0\xmapsto{\,W_1\,}X_1\xmapsto{\,W_2\,}\cdots\xmapsto{\,W_r\,}X_r$ is a finite sequence in which every $W_j$ is a genuine chronological continuation from the actual complete state $X_{j-1}$ to the actual complete state $X_j$.  Then $W:=W_rW_{r-1}\cdots W_1$ is one finite genuine chronological routing word and $X_r=W(X_0)$. Moreover,
\[
\wt_{\mathrm{route}}(W)
=
\sum_{j=1}^{r}\wt_{\mathrm{route}}(W_j),
\]
and this is exactly the total number of genuine Lie-bracket increments in the concatenated chronological history.
\end{lemma}

\begin{proof}
Each continuation begins from the actual complete state produced by all preceding continuations.  Hence their elementary routing operations concatenate in chronological order, and inductively $X_j=W_j\cdots W_1(X_0)$.  No PBW summand from an unrelated history is used. Since routing weight is the number of genuine Lie-bracket increments, these counts add under concatenation. Hence
\[
\wt_{\mathrm{route}}(W)
=
\sum_{j=1}^{r}\wt_{\mathrm{route}}(W_j),
\]
and this sum is exactly the total number of genuine Lie-bracket increments in the concatenated chronological history.
\end{proof}

We can now iterate the local reachability theorem without losing chronological genuineness. At each nonterminal stage the occupancy deficit decreases by at least one, while all previously certified occupancy is retained. Since the deficit is a nonnegative integer bounded by \((p-1)d\), only finitely many such stages are possible.

\begin{theorem}
\label{thm:general-p-global-pbw-completion}
Let \(\mathcal C=(c_1,\ldots,c_d)\) be a fixed ordered homogeneous basis satisfying Theorem~\ref{thm:pbw-reachability}, and let $\nu^{(0)}$ be a completed certified restricted PBW state relative to \(\mathcal C\). Then there exist $r\geq0$ and completed certified states $\nu^{(0)},\ldots,\nu^{(r)}$
such that every transition $\nu^{(j)}\rightsquigarrow\nu^{(j+1)}$ is a finite genuine chronological continuation, every previously certified occupancy is preserved, and
$\Delta_p(\nu^{(0)}) > \Delta_p(\nu^{(1)}) > \cdots > \Delta_p(\nu^{(r)}) = 0$. Moreover, $r\leq\Delta_p(\nu^{(0)})\leq(p-1)d$, and \(\nu^{(r)}\) is terminal. Thus terminal completion requires at most $\Delta_p(\nu^{(0)})\leq(p-1)d$ finite genuine chronological continuation stages.
\end{theorem}

\begin{proof}
If $\Delta_p(\nu^{(j)})>0$, Proposition~\ref{prop:general-p-terminal-characterization} shows that $\nu^{(j)}$ is nonterminal. Theorem~\ref{thm:pbw-reachability} therefore gives a finite genuine chronological continuation to a completed certified state $\nu^{(j+1)}$ satisfying
$\Delta_p(\nu^{(j+1)})\leq\Delta_p(\nu^{(j)})-1$,
while preserving every previously certified occupancy. Since $\Delta_p$ is a nonnegative integer, the process terminates after at most $\Delta_p(\nu^{(0)})$ stages. In particular,
\[
    r\leq\Delta_p(\nu^{(0)})\leq(p-1)d.
\]
At the final stage $\Delta_p(\nu^{(r)})=0$, so Proposition~\ref{prop:general-p-terminal-characterization} gives that $\nu^{(r)}$ is terminal. The bound $\Delta_p(\nu^{(0)})\leq(p-1)d$ is Definition~\ref{def:general-p-local-deficit}.
\end{proof}

Global completion produces the terminal restricted PBW monomial relative to the routing basis \(\mathcal C\). The terminal PBW reduction established earlier, however, is formulated relative to the fixed basis \(\mathcal B=(e_1,\ldots,e_d)\). We must therefore compare the two terminal monomials. Only their top PBW-length components matter. That component is one-dimensional, and the change of homogeneous basis acts on it by the \((p-1)\)-st power of the determinant. Over \(\F_p\) this scalar is \(1\), so the terminal monomial is independent of the chosen ordered homogeneous basis at top PBW length.

\begin{lemma}
\label{lem:general-p-terminal-pbw-rebasing}
Let $\mathcal B=(e_1,\ldots,e_d)$ be the fixed ordered homogeneous basis of $E$, put $\Omegatopp:=e_1^{p-1}\cdots e_d^{p-1}$, and let $\mathcal C=(c_1,\ldots,c_d)$ be any ordered homogeneous basis of $E$.  If $\Omega_{\mathcal C}^{(p)}:=c_1^{p-1}\cdots c_d^{p-1}$, then $[\Omegatopp]\Omega_{\mathcal C}^{(p)}=1$. Equivalently, $\Omega_{\mathcal C}^{(p)}=\Omegatopp+$ terms of PBW length at most $(p-1)d-1$.
\end{lemma}

\begin{proof}
Pass to the associated graded algebra for PBW length.  By restricted PBW, $\gr_{\ell}u(E)\cong S(E)/\langle x^p:x\in E\rangle$. Its top nonzero homogeneous component has degree $(p-1)d$ and is one-dimensional.  Relative to $\mathcal B$ it is spanned by $e_1^{p-1}\cdots e_d^{p-1}$, and relative to $\mathcal C$ by $c_1^{p-1}\cdots c_d^{p-1}$.  Let $A\in\GL_d(\F_p)$ be the change-of-basis matrix from $\mathcal B$ to $\mathcal C$.  On the top component the induced action is multiplication by $(\det A)^{p-1}$.  Since $\det A\in\F_p^\times$, Fermat's theorem gives $(\det A)^{p-1}=1$.  Hence the top PBW-length symbol of $\Omega_{\mathcal C}^{(p)}$ is exactly the symbol of $\Omegatopp$.
\end{proof}

We can now return to the chronological construction. Starting from the empty certificate, global completion reaches terminal occupancy in the routing basis. Concatenation makes the successive continuations into one genuine routing word, and the rebasing lemma then transfers the resulting nonzero terminal coefficient to the fixed terminal PBW direction \(\Omegatopp\).

\begin{theorem}
\label{thm:general-p-terminal-routing}
There exists one finite genuine chronological routing word $W$ such that $[\Omegatopp]W(1)\neq0$. Equivalently, $\Omegatopp=e_1^{p-1}\cdots e_d^{p-1}$ occurs with nonzero coefficient in the complete graded output of one genuine chronological routing history.
\end{theorem}

\begin{proof}
Begin with $X_0:=1$ and the empty completed certificate $\nu^{(0)}=(0,\ldots,0)$ relative to the fixed ordered homogeneous basis $\mathcal C=(c_1,\ldots,c_d)$ supplied by Theorem~\ref{thm:pbw-reachability}. It is certified because $[1]_{\mathcal C}1=1$. Its occupancy vector is $(0,\ldots,0)$, hence $\Delta_p(\nu^{(0)})=(p-1)d$. Apply
Theorem~\ref{thm:general-p-global-pbw-completion}. We obtain completed certified states $\nu^{(0)}\rightsquigarrow\nu^{(1)}\rightsquigarrow\cdots\rightsquigarrow\nu^{(r)}$,
$r\leq(p-1)d$, carried by complete outputs $X_0,\ldots,X_r$, with $\Delta_p(\nu^{(r)})=0$. Let $W_j$ realize the continuation from $X_{j-1}$ to $X_j$. By
Lemma~\ref{lem:general-p-concatenation-genuine-continuations}, $W:=W_r\cdots W_1$ is one genuine chronological routing word and $X_r=W(1)$. Since $\Delta_p(\nu^{(r)})=0$,
Proposition~\ref{prop:general-p-terminal-characterization} gives $\nu^{(r)}=(p-1,\ldots,p-1)$. Certification gives some $\lambda\in\F_p^\times$ such that 
$[c_1^{p-1}\cdots c_d^{p-1}]_{\mathcal C}X_r=\lambda$. There is exactly one restricted PBW monomial of PBW length $(p-1)d$ relative to $\mathcal C$, namely $c_1^{p-1}\cdots c_d^{p-1}$. Thus the top PBW-length component of $X_r$ is $\lambda c_1^{p-1}\cdots c_d^{p-1}$. By Lemma~\ref{lem:general-p-terminal-pbw-rebasing}, its $\Omegatopp$-coefficient is $\lambda$. All remaining terms have smaller PBW length and cannot contribute to $\Omegatopp$. Therefore $[\Omegatopp]W(1)=[\Omegatopp]X_r=\lambda\neq0$.
\end{proof}

Theorem~\ref{thm:general-p-terminal-routing}, together with the terminal
PBW reduction, completes the constructive part of the argument. It shows
that the lower Lie series survives to the extremal degree detected by the
upper Lie series, and therefore gives
\[
    t_L(\F_pG)\geq N_p+2=t^L(\F_pG).
\]
It remains only to prove the reverse inequality. This is formal and does
not use the routing construction: for every unital associative algebra
\(A\), the lower Lie powers are contained in the corresponding upper Lie
powers.

Indeed, for every \(n\geq1\),
\(A^{[n]}\subseteq A^{(n)}\). Every left-normed \(n\)-fold Lie
commutator belongs to \(A^{(n)}\). This follows by induction on \(n\):
the assertion is immediate for \(n=1\), and if
\([x_1,\ldots,x_n]\in A^{(n)}\), then
\[
    [x_1,\ldots,x_n,x_{n+1}]
    =
    [[x_1,\ldots,x_n],x_{n+1}]
    \in
    [A^{(n)},A]
    \subseteq
    A^{(n+1)}.
\]
Since \(A^{(n)}\) is a two-sided ideal, it therefore contains the
two-sided ideal generated by all \(n\)-fold Lie commutators, namely
\(A^{[n]}\). Consequently, whenever \(t^L(A)<\infty\),
\(t_L(A)\leq t^L(A)\): if \(m=t^L(A)\), then
\(A^{(m)}=0\), and hence \(A^{[m]}\subseteq A^{(m)}=0\).

Applying this to \(A=\F_pG\) gives \(t_L(\F_pG)\leq t^L(\F_pG)\). Together with the lower bound obtained from terminal routing, \(t_L(\F_pG)\geq t^L(\F_pG)\), the two indices therefore coincide. The previously established value of the upper index identifies their common value.

\begin{theorem}
\label{thm:main-common-p}
Let \(p\) be a prime and let \(G\) be a finite \(p\)-group. Then
\[
    t_L(\F_pG)
    =
    t^L(\F_pG)
    =
    N_p+2.
\]
\end{theorem}

\begin{proof}
The terminal routing construction gives $t_L(\F_pG)\geq t^L(\F_pG)$, whereas the formal containment $(\F_pG)^{[n]}\subseteq(\F_pG)^{(n)}$ $(n\geq1)$
gives $t_L(\F_pG)\leq t^L(\F_pG)$. Hence $t_L(\F_pG)=t^L(\F_pG)$. By the Jennings upper-index formula established in Section~\ref{sec:introduction},
$t^L(\F_pG)=N_p+2$, and therefore $t_L(\F_pG)=t^L(\F_pG)=N_p+2$.
\end{proof}

It remains to pass from the prime field to the finite-field setting of the main theorem. Let \(K=\F_{p^k}\), put \(R=\F_pG\), and identify \(KG\cong K\otimes_{\F_p}R\). The required passage is formal: scalar extension commutes with both the lower and upper Lie-power constructions, and faithful scalar extension preserves whether each term is zero. Consequently the first vanishing term of each Lie series is unchanged.

\begin{lemma}
\label{lem:lie-powers-scalar-extension}
For every \(n\geq1\), $(K\otimes_{\F_p}R)^{[n]}=K\otimes_{\F_p}R^{[n]}$, $(K\otimes_{\F_p}R)^{(n)}=K\otimes_{\F_p}R^{(n)}$.
\end{lemma}

\begin{proof}
Multilinearity gives
\[
    [\lambda_1\otimes x_1,\ldots,\lambda_n\otimes x_n]
    =
    (\lambda_1\cdots\lambda_n)\otimes[x_1,\ldots,x_n],
\]
so the two-sided ideal generated by the \(n\)-fold commutators is
\[
    (K\otimes_{\F_p}R)^{[n]}
    =
    K\otimes_{\F_p}R^{[n]}.
\]

For the upper Lie powers, induction on \(n\) gives
\[
\begin{aligned}
    (K\otimes_{\F_p}R)^{(n+1)}
    &=
    (K\otimes_{\F_p}R)
    [K\otimes_{\F_p}R^{(n)},K\otimes_{\F_p}R]
    (K\otimes_{\F_p}R)\\
    &=
    K\otimes_{\F_p}\bigl(R[R^{(n)},R]R\bigr)\\
    &=
    K\otimes_{\F_p}R^{(n+1)},
\end{aligned}
\]
with \(n=1\) immediate.
\end{proof}

Thus every term of each Lie filtration over \(KG\) is obtained by scalar
extension from the corresponding term over \(\F_pG\). It remains only to
translate these identities into equality of the first vanishing indices. Since extension of scalars from \(\F_p\) to \(K\) is faithful, an
\(\F_p\)-vector space is zero exactly when its scalar extension to \(K\)
is zero. Hence, for every \(n\geq1\), the \(n\)-th lower Lie power of
\(KG\) vanishes exactly when the \(n\)-th lower Lie power of \(\F_pG\)
vanishes, and the same is true for the corresponding upper Lie powers.
Thus the first vanishing term of each Lie series is unchanged by scalar
extension. Consequently,
\(t_L(KG)=t_L(\F_pG)\) and \(t^L(KG)=t^L(\F_pG)\). Combining these
identities with Theorem~\ref{thm:main-common-p} gives
\(t_L(KG)=t^L(KG)=N_p+2\).

This completes the argument. The routing construction supplies a genuine
chronological history reaching the terminal PBW direction, and hence the
required lower bound for the lower Lie nilpotency index. The general
containment of lower Lie powers in upper Lie powers supplies the reverse
inequality, while scalar extension shows that the resulting equality over
\(\F_p\) persists over every finite field \(K\) of characteristic \(p\).



\section{Conclusion}
\label{sec:conclusion}

We have proved that, for every finite field \(K\) of characteristic \(p\)
and every finite \(p\)-group \(G\), \(t_L(KG)=t^L(KG)=N_p+2\). Thus the
lower and upper Lie nilpotency indices of a modular group algebra
coincide without any additional structural hypothesis on the derived
subgroup of \(G\). The proof converts the static information supplied
by the upper Lie dimension filtration into one genuine chronological
lower Lie history. Jennings' formula identifies the terminal restricted
PBW weight \(N_p\), while the routing construction organizes the
homogeneous directions through a fixed basis with stored PUSH and POWER
provenance. The protection and noncancellation machinery ensures that
distinguished PBW contributions survive competing routing terms and
restricted normalization. Reconstruction, FILL, and POWER are then
established simultaneously by strong induction on homogeneous weight,
providing the local chronological operations needed for the subsequent
construction.

These local operations are assembled by BUILD and REBUILD. BUILD creates
new certified occupancy along stored ancestry, while REBUILD restores
any previously certified occupancy spent in producing that progress.
The resulting PBW reachability theorem strictly decreases the occupancy
deficit while preserving all previously certified occupancy. Iterating
this process reaches the terminal restricted PBW state in finitely many
steps, and the successive continuations concatenate to one genuine
chronological routing history. Terminal PBW rebasing then shows that the
fixed terminal monomial \(\Omegatopp\) occurs with nonzero coefficient
in the complete graded output of that single history. The terminal PBW
reduction therefore gives the difficult inequality
\(t_L(\F_pG)\geq t^L(\F_pG)\), while the reverse inequality follows
formally from \(A^{[n]}\subseteq A^{(n)}\). Hence
\(t_L(\F_pG)=t^L(\F_pG)=N_p+2\). Finally, scalar extension preserves
both Lie filtrations and their first vanishing terms, carrying this
equality from \(\F_pG\) to \(KG\).

The construction is uniform in \(p\). In particular, the same restricted
PBW, stored ancestry, protection, simultaneous reconstruction, and
BUILD and REBUILD mechanisms apply for every prime, including
characteristics \(2\) and \(3\). The argument therefore realizes the
Jennings upper Lie weight by an actual chronological lower Lie process
within a single uniform construction, and yields
\(t_L(KG)=t^L(KG)\) for every finite \(p\)-group \(G\) over every finite
field \(K\) of characteristic \(p\).

\bibliographystyle{plain}
\bibliography{references}

@article{BhandariPassi1992,
 author = {Bhandari, Ashwani K. and Passi, I. B. S.},
 title = {Lie-nilpotency indices of group algebras},
 fjournal = {Bulletin of the London Mathematical Society},
 journal = {Bull. Lond. Math. Soc.},
 issn = {0024-6093},
 volume = {24},
 number = {1},
 pages = {68--70},
 year = {1992},
 language = {English},
 doi = {10.1112/blms/24.1.68},
 zbMATH = {58500},
 Zbl = {0777.20004}
}

@article{BhattChandra2022,
 author = {Bhatt, Suchi and Chandra, Harish},
 title = {A note on modular group algebras with upper {Lie} nilpotency indices},
 fjournal = {Algebra and Discrete Mathematics},
 journal = {Algebra Discrete Math.},
 issn = {1726-3255},
 volume = {33},
 number = {2},
 pages = {68--70},
 year = {2022},
 language = {English},
 url = {admjournal.luguniv.edu.ua/index.php/adm/article/view/1694},
 zbMATH = {7614198},
 Zbl = {1512.16022}
}

@article{BovdiJuhaszSpinelli2006,
 author = {Bovdi, Victor and Juh{\'a}sz, Tibor and Spinelli, Ernesto},
 title = {Modular group algebras with almost maximal {Lie} nilpotency indices.},
 fjournal = {Algebras and Representation Theory},
 journal = {Algebr. Represent. Theory},
 issn = {1386-923X},
 volume = {9},
 number = {3},
 pages = {259--266},
 year = {2006},
 language = {English},
 doi = {10.1007/s10468-006-9022-5},
 url = {link.springer.com/10.1007/s10468-006-9022-5},
 zbMATH = {5125407},
 Zbl = {1115.16012}
}

@article{BovdiSpinelli2005,
 author = {Bovdi, Victor and Spinelli, Ernesto},
 title = {Modular group algebras with maximal {Lie} nilpotency indices.},
 fjournal = {Publicationes Mathematicae Debrecen},
 journal = {Publ. Math. Debr.},
 issn = {0033-3883},
 volume = {65},
 number = {1-2},
 pages = {243--252},
 year = {2004},
 language = {English},
 doi = {10.5486/pmd.2004.3234},
 zbMATH = {2186377},
 Zbl = {1070.16026}
}

@article{BovdiSrivastava2010,
 author = {Bovdi, V. and Srivastava, J. B.},
 title = {Lie nilpotency indices of modular group algebras.},
 fjournal = {Algebra Colloquium},
 journal = {Algebra Colloq.},
 issn = {1005-3867},
 volume = {17},
 number = {1},
 pages = {17--26},
 year = {2010},
 language = {English},
 doi = {10.1142/S1005386710000040},
 url = {www.worldscinet.com/ac/17/1701/S1005386710000040.html},
 zbMATH = {5658023},
 Zbl = {1197.16030}
}

@book{Dixonetall1999,
 author = {Dixon, J. D. and du Sautoy, M. P. F. and Mann, A. and Segal, D.},
 title = {Analytic pro-{{\(p\)}} groups. {Revised} and enlarged by {Marcus} du {Sautoy} and {Dan} {Segal}.},
 edition = {2nd ed.},
 fseries = {Cambridge Studies in Advanced Mathematics},
 series = {Camb. Stud. Adv. Math.},
 volume = {61},
 isbn = {0-521-65011-9},
 year = {1999},
 publisher = {Cambridge: Cambridge University Press},
 language = {English},
 zbMATH = {1339096},
 Zbl = {0934.20001}
}

@article{Hall1934,
 author = {Hall, P.},
 title = {A contribution to the theory of groups of prime-power order},
 fjournal = {Proceedings of the London Mathematical Society. Second Series},
 journal = {Proc. Lond. Math. Soc. (2)},
 issn = {0024-6115},
 volume = {36},
 pages = {29--95},
 year = {1933},
 language = {English},
 doi = {10.1112/plms/s2-36.1.29},
 zbMATH = {3010343},
 Zbl = {0007.29102}
}

@book{Jacobson1962,
 author = {Jacobson, Nathan},
 title = {Lie algebras},
 fseries = {Interscience Tracts in Pure and Applied Mathematics},
 series = {Intersci. Tracts Pure Appl. Math.},
 volume = {10},
 year = {1962},
 publisher = {Interscience Publishers, New York, NY},
 language = {English},
 zbMATH = {3196329},
 Zbl = {0121.27504}
}

@article{Jennings1941,
 author = {Jennings, S. A.},
 title = {The structure of the group ring of a {{\(p\)}}-group over a modular field},
 fjournal = {Transactions of the American Mathematical Society},
 journal = {Trans. Am. Math. Soc.},
 issn = {0002-9947},
 volume = {50},
 pages = {175--185},
 year = {1941},
 language = {English},
 doi = {10.2307/1989916},
 zbMATH = {3040556},
 Zbl = {0025.24401}
}

@book{Passi1979,
 author = {Passi, Inder Bir S.},
 title = {Group rings and their augmentation ideals},
 fseries = {Lecture Notes in Mathematics},
 series = {Lect. Notes Math.},
 issn = {0075-8434},
 volume = {715},
 year = {1979},
 publisher = {Springer, Cham},
 language = {English},
 doi = {10.1007/bfb0067186},
 zbMATH = {3629200},
 Zbl = {0405.20007}
}

@article{PassiSehgal1975,
 author = {Passi, Inder Bir S. and Sehgal, Sudarshan K.},
 title = {Lie dimension subgroups},
 fjournal = {Communications in Algebra},
 journal = {Commun. Algebra},
 issn = {0092-7872},
 volume = {3},
 pages = {59--73},
 year = {1975},
 language = {English},
 doi = {10.1080/00927877508822032},
 zbMATH = {3500039},
 Zbl = {0319.20055}
}

@article{SahaiSharan2018,
 author = {Sahai, Meena and Sharan, Bhagwat},
 title = {On {Lie} nilpotent modular group algebras},
 fjournal = {Communications in Algebra},
 journal = {Commun. Algebra},
 issn = {0092-7872},
 volume = {46},
 number = {3},
 pages = {1199--1206},
 year = {2018},
 language = {English},
 doi = {10.1080/00927872.2017.1339059},
 zbMATH = {6880757},
 Zbl = {1405.16034}
}

@article{SahaiSharan2019,
 author = {Sahai, Meena and Sharan, Bhagwat},
 title = {A note on {Lie} nilpotent group algebras},
 fjournal = {Journal of Algebra and its Applications},
 journal = {J. Algebra Appl.},
 issn = {0219-4988},
 volume = {18},
 number = {9},
 pages = {12},
 note = {Id/No 1950163},
 year = {2019},
 language = {English},
 doi = {10.1142/S0219498819501639},
 zbMATH = {7098687},
 Zbl = {1450.16017}
}

@article{SahaiSharan2021,
 author = {Sahai, Meena and Sharan, Bhagwat},
 title = {On the lower {Lie} nilpotency index of a group algebra},
 fjournal = {Quaestiones Mathematicae},
 journal = {Quaest. Math.},
 issn = {1607-3606},
 volume = {44},
 number = {10},
 pages = {1347--1354},
 year = {2021},
 language = {English},
 doi = {10.2989/16073606.2020.1795740},
 zbMATH = {7462348},
 Zbl = {1489.16023}
}

@article{SahaiSharmaSiwach2017,
 author = {Siwach, Reetu and Sharma, R. K. and Sahai, Meena},
 title = {On the {Lie} nilpotency indices of modular group algebras},
 fjournal = {Beitr{\"a}ge zur Algebra und Geometrie},
 journal = {Beitr. Algebra Geom.},
 issn = {0138-4821},
 volume = {58},
 number = {2},
 pages = {355--367},
 year = {2017},
 language = {English},
 doi = {10.1007/s13366-016-0308-4},
 zbMATH = {6766172},
 Zbl = {1387.16020}
}

@article{SahaiSiwachSharma2018,
 author = {Sahai, Meena and Siwach, Reetu and Sharma, R. K.},
 title = {Lie nilpotency indices of modular group algebras. {II}},
 fjournal = {Asian-European Journal of Mathematics},
 journal = {Asian-Eur. J. Math.},
 issn = {1793-5571},
 volume = {11},
 number = {6},
 pages = {17},
 note = {Id/No 1850087},
 year = {2018},
 language = {English},
 doi = {10.1142/S1793557118500870},
 zbMATH = {6994660},
 Zbl = {1405.16035}
}

@article{Shalev1991,
 author = {Shalev, Aner},
 title = {Lie dimension subgroups, {Lie} nilpotency indices, and the exponent of the group of normalized units},
 fjournal = {Journal of the London Mathematical Society. Second Series},
 journal = {J. Lond. Math. Soc., II. Ser.},
 issn = {0024-6107},
 volume = {43},
 number = {1},
 pages = {23--36},
 year = {1991},
 language = {English},
 doi = {10.1112/jlms/s2-43.1.23},
 zbMATH = {4095670},
 Zbl = {0669.16005}
}

@article{Spinelli2009,
 author = {Spinelli, Ernesto},
 title = {Lie dimension subgroups and central series related to group algebras.},
 fjournal = {Algebra Colloquium},
 journal = {Algebra Colloq.},
 issn = {1005-3867},
 volume = {16},
 number = {3},
 pages = {427--436},
 year = {2009},
 language = {English},
 doi = {10.1142/S1005386709000418},
 zbMATH = {5580836},
 Zbl = {1183.16020}
}

@book{StradeFarnsteiner1988,
 author = {Strade, Helmut and Farnsteiner, Rolf},
 title = {Modular {Lie} algebras and their representations},
 fseries = {Pure and Applied Mathematics, Marcel Dekker},
 series = {Pure Appl. Math., Marcel Dekker},
 volume = {116},
 isbn = {0-8247-7594-5},
 year = {1988},
 publisher = {New York etc.: Marcel Dekker, Inc.},
 language = {English},
 zbMATH = {192889},
 Zbl = {0648.17003}
}

\end{document}